\documentclass[preprint]{elsarticle}
\usepackage{amsthm,amsmath,amssymb,natbib}
\usepackage[colorlinks=true,citecolor=black,linkcolor=black,urlcolor=blue]{hyperref}
\usepackage{tikz}
\usetikzlibrary{calc}
\usetikzlibrary{positioning}

\theoremstyle{plain}
\newtheorem{theorem}{Theorem}
\newtheorem{lemma}[theorem]{Lemma}
\newtheorem{corollary}[theorem]{Corollary}

\newtheorem{fact}[theorem]{Fact}

\theoremstyle{definition}
\newtheorem{definition}[theorem]{Definition}

\theoremstyle{remark}

\makeatletter
\def\ps@pprintTitle{%
 \let\@oddhead\@empty
 \let\@evenhead\@empty
 \def\@oddfoot{}%
 \let\@evenfoot\@oddfoot}
\makeatother

\begin{document}
\begin{frontmatter}

\title{Properties of the Fibonacci-sum graph}

\author[umm]{Andrii Arman}
\ead{armana@myumanitoba.ca}

\author[umm]{David S. Gunderson\corref{dsg}\fnref{fn1}}
\ead{david.gunderson@umanitoba.ca}

\author[umcs]{Pak Ching Li}
\ead{Ben.Li@umanitoba.ca}

\fntext[fn1]{Research supported in part by the Natural
 Sciences and Engineering Research Council of Canada}
 \address[umm]{Department of Mathematics
 University of Manitoba,
 Winnipeg, Manitoba, Canada}
 \address[umcs]{Department of Computer Science,
 University of Manitoba,
 Winnipeg, Manitoba, Canada}

\begin{abstract}
For each positive integer $n$,  the Fibonacci-sum graph $G_n$ on
vertices $1,2,\ldots,n$ is defined by two vertices forming an edge if and only if
they sum to a Fibonacci number. It is known that each $G_n$ is bipartite,
and all Hamiltonian paths in each $G_n$ have been classified. In this paper,
it is shown that each $G_n$ has at most one non-trivial automorphism,
which is given explicitly.  Other properties of $G_n$ are also found,
including the degree sequence, the treewidth,  the nature of the bipartition, and that $G_n$ is
outerplanar.
\end{abstract}


\end{frontmatter}

%

\section{Introduction}
The Fibonacci sequence $\{F_n\}_{n\geq 0}$ is defined by 
$F_0=0, F_1 = 1$ and
for $n\geq 2$, $F_n = F_{n-1}+F_{n-2}$.  Each $F_i$ is called a 
{\em Fibonacci number}. As defined in \cite{FKMOP:14}, 
for each $n\geq 1$, the {\em Fibonacci-sum graph} $G_n=(V,E)$  is the
 graph defined on vertex set
$V=[n]=\{1,2,\ldots,n\}$ with edge set
\[E=\{\{i,j\}:i,j\in V, i\neq j,\  i+j\mbox{ is a Fibonacci number}\}.\]
By this definition, each $G_n$ is a simple graph (no loops or multiple edges).
For example, $G_6$ is depicted in Figure \ref{fi:G6}.
\begin{figure}[h]
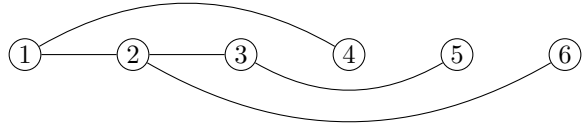
\center
\tikzpicture[scale=0.2]
	\tikzstyle{vertex}=[circle, draw=black,  minimum size=12pt,inner sep=0pt]
	
	\node[vertex] (one) {$1$};
	\node[vertex] (two) [right = of one] {$2$};
	\node[vertex] (three) [right = of two] {$3$};
	\node[vertex] (four) [right = of three] {$4$};
	\node[vertex] (five) [right = of four] {$5$};
	\node[vertex] (six) [right = of five] {$6$};
        \draw (one) to (two);
	\draw (two) to (three);
	\draw (three) to [bend right = 30] node[anchor=south]{}(five);
	\draw (two) to [bend right = 30] node[anchor=south]{}(six);
	\draw (one) to [bend left = 30] node[anchor=south]{}(four);
\endtikzpicture
\caption{The Fibonacci-sum graph $G_6$}\label{fi:G6}
\end{figure}

Inspired by a question of Barwell \cite{Barw:06} about $G_{34}$, K. Fox, Kinnersley, McDonald, Orlow, and 
Puleo \cite{FKMOP:14} classified those Fibonacci-sum graphs that have Hamiltonian paths,
gave a description of such paths, and examined generalizations to other related
graphs. Some of these results are reviewed in Section \ref{ss:hampaths},
and some basic properties of Fibonacci numbers are given in Section \ref{ss:fibnofacts}. 

The main results in this paper are about
additional properties of Fibonacci-sum graphs.  Fibonacci-sum graphs are
connected; this simple result was also observed by Costain \cite[p.~26]{Cost:08},
and is repeated here in Corollary \ref{co:Gnconnected}.
In Section \ref{se:degrees}, degrees of vertices in Fibonacci-sum graphs are 
made explicit. In Section \ref{se:bipartite}, it is shown that every $G_n$ is bipartite
(Theorem \ref{th:bipartite}, first  observed in \cite{Silv:77} and in a more 
general setting, in \cite{AEH:78}), 
and the nature of such bipartitions is examined in more detail. 
The structure of cycles and chords is examined in Section \ref{se:cycles}.

In Section \ref{se:treewidth}, it is shown that for all $n\geq 7$,
 the Fibonacci-sum graph
$G_n$ has treewidth 2.
 In Section \ref{se:planarity}, each $G_n$ is found to be outerplanar (Theorem \ref{th:outerplanar}). In 
Section \ref{se:automorphisms}, for each $n\geq 9$, 
 all automorphisms of 
 $G_n$ are given, and in each case, there are at most two.
 (Section \ref{se:automorphisms} is long since the proof breaks down
 into 11 cases, each case requiring some detailed checking.)
 
\section{Background}\label{se:background}
\subsection{Hamiltonian paths}\label{ss:hampaths}
 In 2006,  Barwell \cite{Barw:06} asked for an ordering of $\{1,2,\ldots, 34\}$ so that
 consecutive pairs sum to a Fibonacci number. Such an ordering corresponds to
 a Hamiltonian path in $G_{34}$.

 In 2014, K. Fox, Kinnersley, McDonald, Orlow, and Puleo \cite{FKMOP:14} answered
 Barwell's question (by giving a recipe for showing how to find the Hamiltonian path) 
 and generalized the result  to whenever $34$ (which is $F_9$)
  is replaced by any Fibonacci number or one less than a Fibonacci number.
 In doing so, they characterized those
 $n$ for which $G_n$ has a Hamiltonian path, and in those cases, they identified all
 such paths. 
 \begin{theorem} [\cite{FKMOP:14}]\label{th:FKMOPmain}
 The Fibonacci-sum graph $G_n$  has
 a Hamiltonian path if and only if  $n$ is either 9, 11, a Fibonacci number, or one less
 than a Fibonacci number.
 The Hamiltonian path is unique
 except when $i\equiv 1\pmod{3}$, $n\in \{F_i,F_i-1\}$, in which case
  there are only two Hamiltonian paths,  paths which agree except on the last three
  vertices (and so the last four vertices form a cycle).
  \end{theorem}
  In the proof of Theorem \ref{th:FKMOPmain}, the authors also showed that each
  $G_n$ had at least one vertex of degree 1, thereby preventing any Hamiltonian
  cycles (this fact is also proved below in Lemma \ref{le:pendants}). 
   
  From Figure~\ref{fi:G6}, it is seen that indeed $G_6$ has no Hamiltonian path.
  Since 34 is a Fibonacci number, 
 Theorem \ref{th:FKMOPmain}  guarantees that $G_{34}$ has a unique
 Hamiltonian path, partially answering the question of Barwell. 
 The same authors \cite{FKMOP:14} also gave ways to find a Hamiltonian path,
 provided one exists: 
\begin{theorem}[see  {\cite[Thm.~1 proof]{FKMOP:14}}]\label{th:FKMOPpath}
For $k\geq 5$, the subgraph of $G_{F_k}$ formed by edges whose sum is in
$\{F_{k-1}, F_k, F_{k+1}\}$ is a Hamiltonian path.
\end{theorem}
 
 All Hamiltonian paths guaranteed by Theorem \ref{th:FKMOPpath} were found:
  \begin{theorem}[see {\cite[Thm.~3]{FKMOP:14}}] \label{th:endpoints}
  Let $k\geq 5$. If $k\equiv 0\pmod{3}$,
  then $G_{F_k}$ contains a unique Hamiltonian path with end vertices $F_k$ and $\frac{F_k}{2}$.
 If $k\equiv 2\pmod{3}$, the unique Hamiltonian path in $G_{F_k}$
 has end vertices $F_k$ and $\frac{F_{k+1}}{2}$.
 If $k\equiv 1\pmod{3}$, there are two Hamiltonian paths, one with endpoints $F_k$ and
 $\frac{F_{k-1}}{2}$, the other with endpoints $F_k$ and $F_k-\frac{F_{k-4}}{2}$.
  \end{theorem}
  
For example,  to find the (unique) Hamiltonian path in $G_{34}$ (asked for by Barwell),
 Theorem  \ref{th:endpoints} says to
use 34 and 17 as endpoints. Starting with 17, and simply applying 
Theorem \ref{th:FKMOPpath}
(using sums only in $\{21,34,55\}$) gives the path: 17, 4,  30, 25,  9, 12,  22, 33, 1,
 20, 14, 7, 27, 28,  6, 15, 19, 2, 32, 23, 11, 10, 24, 31, 3, 18, 16, 5, 29, 26, 8,  
 13, 21, 34.
(At each point in the algorithm, the next vertex is uniquely determined. )
  
   As an example where two Hamiltonian paths exist (with $k=7$ in the statement of
  Theorem \ref{th:FKMOPpath})  in $G_{13}$, using only
  sums in $\{8,13,21\}$, the two  paths are: 13, 8, 5, 3, 19, 11, 2, 6, 7, 1, 12, 9,  4 and
  13,  8,  5, 3, 19, 11, 2,  6, 7, 1, 4,  9, 12. \medskip
  
  \subsection{Fibonacci number facts}\label{ss:fibnofacts}
  Here are some well-known facts about Fibonacci numbers that are used below.
\begin{fact}\label{fa:evenfibs}
The Fibonacci number $F_k$ is even iff $k\equiv 0\pmod{3}$.
\end{fact}

\begin{fact}\label{fa:sumevenfibs}
 For each $\ell\geq 0$,
 \[\sum_{i=0}^\ell F_{2i}=F_{2\ell+1}-1.\]
 \end{fact}

\begin{fact} For each $k\geq 3$,
\[\frac{F_{k-3}+F_k}{2}=F_{k-1}.\]
\end{fact}

\begin{fact} \label{fa:halffib}
For $k\geq 0$, $F_k<\frac{F_{k+2}}{2}<F_{k+1}$.
\end{fact}

\begin{fact}\label{fa:twoFssumtoF}
The only pairs of Fibonacci numbers whose sum is another Fibonacci
number are the consecutive pairs. That is, if for some $i,j,k$, $F_i+F_j=F_k$,
then $j=i+1$ and $k=i+2$ or $j=i-1$ and $k=i+1$.
\end{fact}

\section{Degrees}\label{se:degrees}
In a graph, vertices of degree 1 are called {\em pendant vertices},
or simply, pendants. When $n$ is a Fibonacci number, the Fibonacci-sum graph
on $n$ vertices  has a pendant.
\begin{lemma}[\cite{FKMOP:14}]\label{le:Fkonenbr}
 For each $k\geq 1$, in $G_{F_k}$, the vertex $F_k$ has
only one neighbour, namely $F_{k-1}$.
\end{lemma}
The basic idea used to prove Lemma \ref{le:Fkonenbr} can be extended;
the following is essentially contained in \cite{FKMOP:14}, but formulated
slightly differently here.
\begin{lemma}\label{le:lastvertex}
Let  $n\geq 2$, and let $k$ be so that $F_k\leq n<F_{k+1}$.
In $G_n$,  vertex $n$ is adjacent to only
\begin{align*} \begin{cases} 
           F_{k+1}-n &\mbox{ if } n\leq \frac{F_{k+2}}{2};\\
           F_{k+1}-n \mbox{ and }F_{k+2} -n&\mbox{ if } n>\frac{F_{k+2}}{2}.
                       \end{cases}
\end{align*}
\end{lemma}
\begin{proof} Let $x\in [1,n]$ be adjacent to $n$; in other words, 
let $i$ be so that $x+n=F_i$. Since
\[F_k<1+F_k\leq x+n\]
and
\[x+n \leq 2n-1<2F_{k+1}<F_{k+1}+F_{k+2}=F_{k+3},\]
it follows that $i\in\{k+1, k+2\}$.

When $n\leq \frac{F_{k+2}}{2}$, $x+n<2n\leq F_{k+2}$,
so $i\neq k+2$. Thus $x=n-F_{k+1}$ is the only possible solution. 
When $n>\frac{F_{k+2}}{2}$,
the possible solution $x=F_{k+2}-n$ is also realizable (with $x>F_k$).
\end{proof}

Lemma \ref{le:lastvertex} and induction imply the following known 
result (see, e.g., \cite{Cost:08}):
\begin{corollary}\label{co:Gnconnected}
For each $n\geq 1$, $G_n$ is connected.
\end{corollary}

The next lemma is also implicit in \cite[proof of Thm.~4]{FKMOP:14}, but is stated
and proved here separately for later reference.
\begin{lemma} \label{le:pendants}
Let $n \geq 2$, and  let $k$ be so that $F_k\leq n<F_{k+1}$. Then in $G_n$, the
vertex $F_k$ has only one neighbour, namely $F_{k-1}$.
\end{lemma}
\begin{proof}   If  $n=F_k$, then
 by Lemma \ref{le:Fkonenbr}, in $G_n$ the vertex $n$ is adjacent to only $F_{k-1}$.
(This also follows from the first case of Lemma \ref{le:lastvertex} 
since by Fact \ref{fa:halffib}, $n< \frac{F_{k+2}}{2}$).

So assume that $F_k < n < F_{k+1}$.
 In $G_{F_{k}}$, the vertex $F_k$ is adjacent to only $F_{k-1}$, and so in $G_n$, the
 vertex $F_{k}$ has only one neighbour smaller than $F_k$.
  To see that in $G_n$, vertex $F_k$ has no larger neighbours,  consider  some
  vertex $v$ satisfying $F_k<v\leq n$. Then
 \[F_{k+1}<F_k+F_k <\ F_k + v \  \leq F_k + n < F_{k} + F_{k+1} = F_{k+2}\]
 shows that  $F_k+v$ is not a  Fibonacci number, and so $v$ is not adjacent to $F_k$.
 Hence in $G_n$, $F_k$ has only one neighbour (namely $F_{k-1}$).
\end{proof}
 
 Calculating the degree of any vertex in a Fibonacci-sum graph
 is relatively straightforward.  To see the idea,
 consider the  adjacency matrix for $G_{18}$ given in Figure \ref{fi:adjG18}, 
 
 \begin{figure}[ht]\center
  \renewcommand{\arraystretch}{1.05}
    \renewcommand{\tabcolsep}{3pt}
 \begin{tabular}{r| p{10pt} |p{10pt}|p{10pt}|p{10pt}|p{10pt}|p{10pt}|p{10pt}|p{10pt}|p{10pt}|p{10pt}|p{10pt}|p{10pt}|p{10pt}|p{10pt}|p{10pt}|p{10pt}|p{10pt}|p{10pt}|}
 
 &$F_2$&$F_3$&$F_4$&&$F_5$&&&$F_6$&&&&&$F_7$&&&&\\ 
 &1&2&3&4&5&6&7&8&9&10&11&12&13&14&15&16&17&18\\ \hline
1    & 0 &1&&1&&&1&&&&&1&&&&&&\\ \hline
2    &1&&1&&&1&&&&&1&&&&&&&\\ \hline
 3   &&1&&&1&&&&&1&&&&&&&&1\\ \hline
 4   &1&&&0&&&&&1&&&&&&&&1 &\\ \hline
5    &&&1&&&&&1&&&&&&&&1&&\\ \hline 
6    &&1&&&&&1&&&&&&&&1&&&\\ \hline
7    &1&&&&&1&&&&&&&&1&&&&\\ \hline
8    &&&&&1&&&&&&&&1&&&&&\\ \hline
9    &&&&1&&&&&&&&1&&&&&&\\ \hline
10  &&&1&&&&&&&&1&&&&&&&\\ \hline
11  &&1&&&&&&&&1&&&&&&&&\\ \hline
12  &1&&&&&&&&1&&&&&&&&&\\ \hline
 13 &&&&&&&&1&&&&&&&&&&\\ \hline
 14 &&&&&&&1&&&&&&&&&&&\\ \hline
 15 &&&&&&1&&&&&&&&&&&&\\ \hline
16  &&&&&1&&&&&&&&&&&&&1\\ \hline
17  &&&&1&&&&&&&&&&&&&0&\\ \hline
18  &&&1&&&&&&&&&&&&&1&&\\ \hline
      \end{tabular}          
\caption{Adjacency matrix for $G_{18}$ (missing entries are 0s).}
\label{fi:adjG18}
\end{figure}
   \bigskip
   
 \begin{theorem}\label{th:degreeformula}
 Let $n\geq 1$ and let $x\in [1,n]$. 
 Let $k\geq 2$ satisfy $F_k\leq x <F_{k+1}$ and $\ell\geq k$ satisfy
 $F_\ell \leq x+n < F_{\ell+1}$.  Then the degree of $x$ in $G_n$
 is 
 \begin{equation}\deg_{G_n}(x)=\begin{cases}
                                  \ell -k &\mbox{if $2x$ is not a Fibonacci number};\\
                                  \ell -k -1&\mbox{if $2x$ is a Fibonacci number}.
                                  \end{cases}\label{eq:degreecount1}
 \end{equation}
 \end{theorem}
 \begin{proof} For each $s\in [n]$, since $F_k< x+s\leq x+n<F_{\ell+1}$, 
 if $x+s$ is a Fibonacci number, then $k<\ell$ and $x+s \in \{F_{k+1}, \ldots, F_{\ell}\}$.
 Thus,
\begin{equation}\label{eq:degcount}
\deg_{G_n}(x)=|\{s\in [n]: s\neq x, x+s\in  \{F_{k+1}, \ldots, F_{\ell}\} \}|,
\end{equation}
and so  $\deg_{G_n}(x)\leq \ell-k$.
 However, when $s=x$,  if $2x$ is a Fibonacci number, since $G_n$ has no loops,
  the sum  $x+x$ fails to generate an edge  and so (\ref{eq:degcount})
  says $\deg_{G_n}(x)= \ell-k-1$.
 Otherwise, if $2x$ is not a Fibonacci number, then (\ref{eq:degcount}) 
 shows $\deg_{G_n}(x)=\ell-k$.
 \end{proof}
 
 The expression (\ref{eq:degreecount1}) can be made more specific:
 
 \begin{theorem}\label{th:degreeformula2}
 Let $n\geq 1$ and let $x\in [1,n]$. 
 Let $k\geq 2$ satisfy $F_k\leq x <F_{k+1}$ and $\ell\geq k$ satisfy
 $F_\ell \leq x+n < F_{\ell+1}$.  Then
 \begin{equation}\deg_{G_n}(x)=\begin{cases}
                      \ell -k -1&\mbox{if $x=1$ or $k\geq 4$ and $x=\frac{1}{2}F_{k+2}$};\\
                        \ell -k &\mbox{otherwise}.
                                  \end{cases}\label{eq:degreecount2}
 \end{equation}
 \end{theorem}
\begin{proof}
Assume that $2x$ is a Fibonacci number.
  Then $2x<2(F_{k+1})<F_{k+3}$, 
 and so either $2x=F_{k+1}$ or $2x=F_{k+2}$.
 
 Case 1: $2x=F_{k+1}$: 
 Then $2F_k\leq 2x=F_{k+1}=F_k+F_{k-1}$ implies 
 $F_k\leq F_{k-1}$, and so $k=2$, $x=s=1$ and $2x=2=F_3=F_{k+1}$.
   One value of
 $s$ in (\ref{eq:degcount}) is lost, and so $\deg_{G_n}(1)=\ell-k-1$.
 
 Case 2: Assume that $k\geq 3$, and let
  $2x=F_{k+2}$. Then $F_{k+2}=x+x\leq x+n<F{\ell+1}$,
 and so $k+2\leq \ell$. Thus 
  one value for $s$ in (\ref{eq:degcount}) is lost and again
  $\deg_{G_n}(\frac{1}{2}F_{k+2})=k-\ell-1$. In fact, when
  $k=3$, $x=2$, and in this case, $2x$ is not a Fibonacci number,
  so one may assume that $k\geq 4$ (and $x\geq 4$).
  
  In all other cases, $2x$ is not a Fibonacci number, so by  (\ref{eq:degcount}),
   $\deg_{G_n}(x)=\ell-k$. Hence  (\ref{eq:degreecount2}) is verified.
  \end{proof}
 
 Examples of degree sequences in $G_n$ are found in the diagrams used in
 the proof of Theorem \ref{th:Automorphism}.
The following  consequence of Theorem \ref{th:degreeformula2}
can be verified (for small $n$) in the adjacency matrix for $G_{18}$ given in
Figure \ref{fi:adjG18}.
\begin{corollary}
For any $n \geq 2$, vertex $2$ has maximum degree in $G_n$.
If $n+1$ is a Fibonacci number, then $\deg_{G_n}(1)=\deg_{G_n}(2)-1$;
otherwise, $\deg_{G_n}(1)=\deg_{G_n}(2)$.
\end{corollary}

Using the degree sequence given in either  Theorem \ref{th:degreeformula}
or \ref{th:degreeformula2}, the number of edges in $G_n$ can be computed;
the proof of the following is left to the reader:
\begin{corollary}
Let $n\geq 1$ and $k\geq 2$ be integers satisfying $F_k\leq n <F_{k+1}$.
Then 
  \begin{equation}\label{eq:numberofedges}
|E({G_n})|=\begin{cases}
n+\frac{F_k+1}{2}-\frac{\left\lfloor\frac{4(k+1)}{3}\right\rfloor}{2}
          & \mbox{if } n\leq \frac{F_{k+2}}{2};\\
2n+\frac{F_k+1}{2}-\frac{\left\lfloor\frac{4(k+1)}{3}\right\rfloor}{2}
                       -\left\lceil\frac{F_{k+2}-1}{2}\right\rceil
          & \mbox{if }n>\frac{F_{k+2}}{2}.
	\end{cases}
	\end{equation}
\end{corollary}
For $n=1,2,\ldots, 21$, the values of $|E(G_n)|$ are 
0, 1, 2, 3, 4, 5, 7, 8, 9, 10, 12, 14, 15, 16, 17, 18, 19, 21, 23, 25, and 26.

Lemma \ref{le:pendants} says that every $G_n$ has at least one pendant vertex
(and so $G_n$ does not have a Hamiltonian cycle). 
 In general, how many pendants can $G_n$ have?
The following observation can be proved directly or by applying Theorems
\ref{th:degreeformula} and \ref{th:degreeformula2}.
\begin{theorem}\label{th:manypendants}
Let $k\geq 3$ and  $n$ satisfy $F_k\leq  n <F_{k+1}$.
If $n<\frac{F_{k+2}}{2}$, then $F_k,F_k+1, \ldots, n$ are the vertices of degree 1.
If $n\geq \frac{F_{k+2}}{2}$, then $F_k, F_k+1, \ldots, F_{k+2}-n-1$ are 
 vertices of degree 1.
The only other possible degree 1 vertices are 
\[\begin{cases}
\frac{F_k}{2} & \mbox{ if } k\equiv 0\pmod{3} \mbox{ and } n<F_{k+1}-\frac{F_k}{2};\\
\frac{F_{k+1}}{2} & \mbox{ if } k\equiv 1\pmod{3};\\
\frac{F_{k+2}}{2} & \mbox{ if } k\equiv 2\pmod{3} \mbox{ and } n\geq \frac{F_{k+2}}{2}.
\end{cases}\]
\end{theorem}
 In classifying automorphisms (see Theorem \ref{th:Automorphism} for 
 diagrams), Theorem \ref{th:manypendants} is useful. 
 For example, in $G_{51}$, using $k=9$ (and $F_9=34$) in Theorem
 \ref{th:manypendants}, $F_{k+2}-n-1=89-51-1=37$, so  
 that vertices 34, 35, 36, 37 and 17  are pendants.

\begin{corollary}
For any $p\geq 1$, there exists an $n$ so that $G_n$ has exactly $p$ pendants.
\end{corollary}
\begin{proof}
For $p\geq 1$, let $n=F_{6p+1}+p-1$. Then, the only degree 1 
vertices in $G_n$  are  $F_{6p+1},F_{6p+1}+1, \ldots,F_{6p+1}+p-1$.
\end{proof}

For example, to find  $n$ so that $G_n$ has exactly $p=3$ pendants,
 let $n=F_{19}+(3-1)=4181+2=4183$.

\section{$G_n$ is bipartite}\label{se:bipartite}
As mentioned in the introduction, it is known (see \cite{Silv:77}, \cite{AEH:78}, or
the survey \cite{Cost:08})
 that $G_n$ is bipartite, but for completeness, the short proof is given here.
\begin{theorem} \label{th:bipartite}
For $n \geq 1$, $G_n$ is bipartite.
\end{theorem}
\begin{proof} The proof is by induction on $n$.  

\noindent{\sc Base step:} For $n\leq 6$, $G_n$ is a tree
(see Figure \ref{fi:G6}) and so is bipartite.

\noindent{\sc Induction step:}
Let $m\geq 6$ and assume that  $G_m$ is bipartite.
 It remains to show that $G_{m+1}$ is bipartite.
 Let $\ell$ be the positive integer such that $F_{\ell} \leq m+1 < F_{\ell+1}$.

Case 1: If  $m+1 \leq \frac{F_{\ell+2}}{2}$, by Lemma \ref{le:lastvertex}, in $G_{m+1}$,
the vertex $m+1$ is adjacent to (only) the vertex $F_{\ell+1}-(m+1)$.
 Since $G_m$ is bipartite, the addition of the single edge to create $G_{m+1}$
 is still bipartite, concluding the inductive step in this case.

Case 2: If If  $m+1 \geq \frac{F_{\ell+2}}{2}$, by Lemma \ref{le:lastvertex},
in $G_{m+1}$, the vertex $m+1$ is adjacent to (only) the two vertices
 $v_1 = F_{\ell+1} - (m+1)$ and $v_2 = F_{\ell+2} - (m+1)$.
By Fact \ref{fa:halffib}, $m+1>F_\ell$, and so the vertex $v=m+1-F_\ell$
is adjacent to both $v_1$ and $v_2$ in $G_m$. By the induction hypothesis,
$G_m$ is bipartite, so $v_1$ and $v_2$ are in the same bipartite class in $G_m$,
in which case the vertex $m+1$ can be added to the opposite class,
showing that $G_{m+1}$ is bipartite. This concludes Case 2 of  the inductive step, 
and hence the proof.
\end{proof}

As an example, Figure \ref{fi:G20bipartite} shows the bipartition of $G_{20}$.

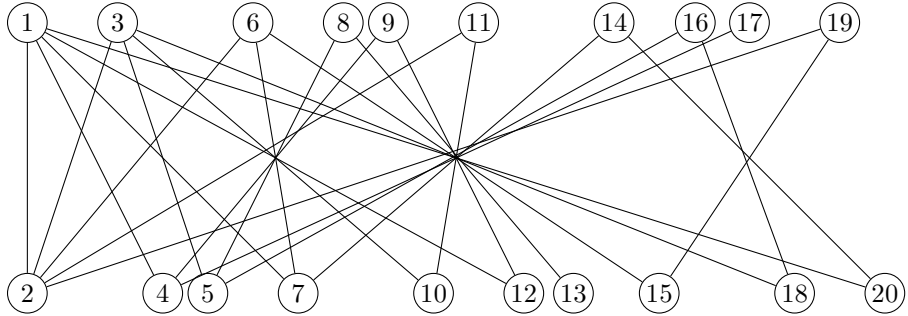
\begin{figure}[h]\centering
\begin{tikzpicture}[scale =0.6]
	\tikzstyle{vertex}=[circle, draw=black,  minimum size=15pt,inner sep=1pt]
	
	\node[vertex] (one) at (1,6){$1$};
	\node[vertex] (three) at (3,6) {$3$};
	\node[vertex] (six) at (6,6) {$6$};
	\node[vertex] (eight) at (8,6) {$8$};
	\node[vertex] (nine) at (9,6) {$9$};
	\node[vertex] (eleven) at (11,6) {$11$};
	\node[vertex] (fourteen) at (14,6){$14$};
	\node[vertex] (sixteen) at (15.8,6) {$16$};
	\node[vertex] (seventeen) at (17,6) {$17$};
	\node[vertex] (nineteen) at (19,6){$19$};
	\node[vertex] (two) at (1,0) {$2$};
	\node[vertex] (four) at (4,0) {$4$};
	\node[vertex] (five) at (5,0) {$5$};
	\node[vertex] (seven) at (7,0) {$7$};
	\node[vertex] (ten) at (10,0) {$10$};
	\node[vertex] (twelve) at (12,0) {$12$};
	\node[vertex] (thirteen) at (13.1,0) {$13$};
	\node[vertex] (fifteen) at (15,0) {$15$};
	\node[vertex] (eighteen) at (18,0) {$18$};
	\node[vertex] (twenty) at (20,0) {$20$};
	
	 \draw (one) to (two);
	  \draw (one) to (four);
	   \draw (one) to (seven);
	    \draw (one) to (twelve);
	    \draw (one) to (twenty);
	  \draw (two) to (three);
	    \draw (two) to (six);
	      \draw (two) to (eleven);
	      \draw (two) to (nineteen);
	   \draw (three) to (five);
	      \draw (three) to (ten);
	        \draw (three) to (eighteen);
	     \draw (four) to (nine);
	       \draw (four) to (seventeen);
	     \draw (five) to (eight);
	       \draw (five) to (sixteen);
	       \draw (six) to (seven);
	    \draw (six) to (fifteen);
	    \draw (seven) to (fourteen);
	    \draw (eight) to (thirteen);
	    \draw (nine) to (twelve);
	    \draw (ten) to (eleven);
	    \draw (fourteen) to (twenty);
	    \draw (fifteen) to (nineteen);
	    \draw (sixteen) to (eighteen);
\end{tikzpicture}
\caption{$G_{20}$ is bipartite}\label{fi:G20bipartite}
\end{figure}

Looking at $G_{20}$ (see Figure \ref{fi:G20bipartite}), one
 might expect that the bipartition of $G_n$ is always nearly balanced.
  However, there are values of  $n$ so that the ``imbalance'' is as large as one wants; before proving this, some lemmas are given.

 \begin{lemma}\label{le:2colourFSgraph}
 Let $n\geq 1$ and (by Theorem \ref{th:bipartite}) let $c:V(G_n)\rightarrow\{0,1\}$
be a 2-colouring defining  the bipartition of $G_n$, and suppose that $c(1)=c(F_1)=1$.
Then for each $k\in [1, \frac{n-1}{2}]$,
\begin{align}
  c(F_{2k})&=1, \mbox{ \  }c(F_{2k+1})=0, \label{eq:colorfibns}\\
 c\left(\frac{F_{6k-3}}{2}\right)&=1, \mbox{ and } c\left(\frac{F_{6k}}{2}\right)=0.\label{eq:colorfibns3mod0}
\end{align}
\end{lemma}

\begin{proof}
Consecutive Fibonacci numbers are adjacent in $G_n$, and so a
 simple inductive argument shows that  (\ref{eq:colorfibns}) holds. 
To see (\ref{eq:colorfibns3mod0}),  by Fact \ref{fa:evenfibs}, starting
with $F_3$, every third Fibonacci number is even and  $c(\frac{F_3}{2})=
c(1)=1$. Since for any $k\geq 1$,
\[\frac{F_{6k}}{2}+\frac{F_{6k-3}}{2}=\frac{1}{2}(F_{6k-1}+F_{6k-2}+F_{6k-3})
=\frac{2F_{6k-1}}{2}=F_{6k-1},\]
pairs of Fibonacci numbers of the form $\{\frac{F_{6k}}{2}, \frac{F_{6k-3}}{2}\}$ 
are adjacent in $G_n$ and  so alternate in colour accordingly.
\end{proof}

\begin{lemma}\label{le:recursioncolour}
Let $k\geq 4$ and let $N$ satisfy $F_k<N<F_{k+1}$. Let $c$ be the 2-colouring
of $G_N$ as in Lemma \ref{le:2colourFSgraph}.
If $N\neq \frac{F_{k+2}}{2}$, then $c(N)=c(N-F_k)$.
If $N=\frac{F_{k+2}}{2}$, then
\[c(N)=c\left(\frac{F_{k+2}}{2}\right)=\begin{cases}
                                          1 & \mbox{if }k\equiv 1\pmod{6}\\
                                          0 & \mbox{if }k\equiv 4\pmod{6}.
                                          \end{cases}\]
\end{lemma}
\begin{proof} If $N=\frac{F_{k+2}}{2}$, then the value for
$c(\frac{F_{k+2}}{2})$ is given by Lemma \ref{le:2colourFSgraph}.
So suppose that  $N \neq \frac{F_{k+2}}{2}$.
Then $F_{k+1}-N \neq N - F_{k}$ and vertex $N$ is adjacent to  $F_{k+1}-N$.
 So, $c(N)=1-c(F_{k+1}-N)$. The vertex $F_{k+1}-N$ is adjacent to  $N-F_{k}$ and therefore $c(N)=1-c(F_{k+1}-N)=1-(1-c(N-F_k))=c(N-F_k)$.
\end{proof}

For the next two lemmas, for $N\geq 1$ define 
\begin{equation}\label{eq:S(N)}
S(N)=\sum_{i=1}^N c(i)-\frac{N}{2},
\end{equation}
 and put $S(0)=0$.

\begin{lemma}\label{le:SFibonacci}
For $k\geq 3$, and $S$ as defined in (\ref{eq:S(N)}), 
\[S(F_{k})=\begin{cases}
    \hspace*{.1in}   0               & \mbox{if }k\equiv 0,3\pmod{6}\\
    \hspace*{.1in}   \frac{1}{2} & \mbox{if }k\equiv 2,4\pmod{6}\\
       -\frac{1}{2} & \mbox{if }k\equiv 1,5\pmod{6}.
       \end{cases}\]
\end{lemma}
\begin{proof}

\noindent{\sc Case 1}: $k\equiv 0,3\pmod{6}$. 

 For  $i \neq \frac{F_{k}}{2}$, vertex $i$ is
 adjacent to $F_{k+1}-i$, so $c(i)+c(F_{k+1}-i)=1$.
\begin{align*}
S(F_{k})&=\sum_{i=1}^{F_k} c(i)-\frac{F_k}{2}=\sum_{i=1}^{\frac{F_k}{2}-1}
\left[c(i)+c(F_{k+1}-i)\right]+c\left(\frac{F_k}{2}\right)+c(F_k)-\frac{F_k}{2}\\
        &=\frac{F_{k}}{2}-1+c\left(\frac{F_k}{2}\right)+c(F_k)-\frac{F_k}{2}
                       =c\left(\frac{F_k}{2}\right)+c(F_k)-1.
\end{align*}
 By Lemmas 
\ref{le:2colourFSgraph} and  \ref{le:recursioncolour},
$c\left(\frac{F_{6l}}{2}\right)+c(F_{6l})-1=0+1-1=0$ and 
$c\left(\frac{F_{6l+3}}{2}\right)+c(F_{6l+3})-1=1+0-1=0$.\bigskip

\noindent{\sc Case 2}: $k\equiv 2,4\pmod{6}$.
\begin{align*}
S(F_{k})&=\sum_{i=1}^{F_k} c(i)-\frac{F_k}{2}=\sum_{i=1}^{\frac{F_k-1}{2}}
\left[c(i)+c(F_{k+1}-i)\right]+c(F_k)-\frac{F_k}{2}\\
        &=\frac{F_{k}-1}{2}+c(F_k)-\frac{F_k}{2}=-\frac{1}{2}+c(F_k)
                        =-\frac{1}{2}+1=\frac{1}{2}.
\end{align*}\bigskip

\noindent{\sc Case 3}: $k\equiv 1,5\pmod{6}$.
\begin{align*}
S(F_{k})&=\sum_{i=1}^{F_k} c(i)-\frac{F_k}{2}
    =\sum_{i=1}^{\frac{F_k-1}{2}}\left[c(i)+c(F_{k+1}-i)\right]+c(F_k)-\frac{F_k}{2}\\
        &=\frac{F_{k}-1}{2}+c(F_k)-\frac{F_k}{2}=-\frac{1}{2}+c(F_k)
                  =-\frac{1}{2}+0=-\frac{1}{2}.
\end{align*}
\end{proof}

To show that the bipartite imbalance is unbounded, a lemma is used:
\begin{lemma}\label{le:SN}
 Let $N$ and $k$ be  integers satisfying $F_{k}<N<F_{k+1}$. 
 With $S(N)$ as defined in (\ref{eq:S(N)}), if $k \equiv 0,2,3,5\pmod{6}$, 
  or if $k\equiv 1,4\pmod{6}$ and $N<\frac{F_{k+2}}{2}$,
   then $S(N)=S(F_k)+S(N-F_k)$.
 
 \end{lemma}
\begin{proof}
\begin{align*}
S(N) &=\sum_{i=1}^N c(i)-\frac{N}{2} &&\mbox{(by def'n)}\\
     &=\sum_{i=1}^{F_k} c(i)-\frac{F_k}{2}+\sum_{i=F_k+1}^N c(i)- \frac{N-F_k}{2}\\
     &=S(F_k)+\sum_{i=1}^{N-F_k} c(i)-\frac{N-F_k}{2}&&\mbox{(by Lemma \ref{le:recursioncolour})}\\
     &=S(F_k)+S(N-F_k), &&\mbox{(by def'n)}
\end{align*}
as desired.
\end{proof}

\begin{theorem}\label{th:imbalance}
For each positive integer $z$, there exists  $n$ so that if $G_n$ has bipartition
$V(G_n)=A\cup B$, then $\left| |A| -\frac{n}{2}\right| = z$.
\end{theorem}
\begin{proof} Let $z\in \mathbb{Z}^{+}$.
Define $n=\sum_{i=1}^{2z} F_{6i+1}$. Then $F_{12z+1}\leq n<F_{12z+2}$
and so $n<\frac{F_{12z+3}}{2}=\frac{F_{(12z+1)+2}}{2}$. Using $k=12z+1\equiv 1\pmod{6}$, by Lemmas  \ref{le:SFibonacci}
and \ref{le:SN}, 
\[ S(n)=S\left(\sum_{i=1}^{2z} F_{6i+1}\right)
         =\sum_{i=1}^{2z} S(F_{6i+1})=2z \cdot \frac{-1}{2}=-z.\]
Assuming that (as in Lemma \ref{le:2colourFSgraph}) $A$ is the set of vertices 
with colour 1, the imbalance of $A$ is 
$\left| |A| -\frac{n}{2}\right|=|S(n)+\frac{n}{2}-\frac{n}{2}|=z$.
\end{proof}

For any $n$, it is possible to determine the sizes of parts in bipartition 
of $G_n$ using techniques of Theorem \ref{th:imbalance} and 
Zeckendorf's \cite{Zeck:72} representation of $n$.
An example covered by Theorem \ref{th:imbalance} with $z=2$ is
when $n=F_{20}=7164$, in which case the bipartition is 
3580 vs  3584.

\section{Cycles in $G_n$} \label{se:cycles}
 The girth of a graph $G$, denoted $\mbox{girth}(G)$, is the length of a 
 shortest cycle in $G$.
As seen in Figure \ref{fi:G6}, if $n\leq 6$ then $G_n$ is acyclic.
\begin{corollary}
For $n \geq 7$, $\mbox{girth}(G_n) = 4$.
\end{corollary}
\begin{proof}
For $n \geq 7$, a 4-cycle in $G_n$ is  $(1,2,6,7)$.
 By Theorem \ref{th:bipartite}, $G_n$ is bipartite and so contains no triangles.
\end{proof}

For large $n$, $G_n$ contains many 4-cycles.  When $F_k\leq n <F_{k+1}$ 
and $n>\frac{F_{k+2}}{2}$,
a  4-cycle is $(n-F_k, F_{k+2}-n, n, F_{k+1}-n)$ as shown Figure \ref{fi:last4cycle},
where edges are labelled with sums.

\begin{figure}[h]\centering
\begin{tikzpicture}

	\tikzstyle{vertex}=[circle, fill=black,  minimum size=6pt,inner sep=0pt]
	\tikzstyle{emptyvertex}=[circle, draw=black, minimum size=5pt,inner sep=0pt]
	\node[vertex,label=below:{$n-F_k$}] at (-5,0) (x1){};
	\node[emptyvertex,label=below:{$F_{k-1}$}] at (-3.5,0) (x2){};
	\node[vertex,label=below:{$F_{k+1}-n$}] at (-2,0) (x3){};
	\node[emptyvertex,label=below:{$F_k$}] at (0,0) (x4){};
	\node[vertex,label=above:{$F_{k+2}-n$}] at (1.5,0) (x5){};
	\node[vertex,label=below:{$n$}] at (4.5,0) (x6){};	
	\node[emptyvertex,label=below:{$F_{k+1}$}] at (6,0) (x7){};	
        \draw (x1) to  [bend left = 30] node[anchor=south]{$F_{k-1}$}(x3);
        \draw (x1) to  [bend right= 30] node[anchor=north]{$F_{k+1}$}(x5);
	\draw(x3) to  [bend left = 40] node[anchor=south]{$F_{k+1}$}(x6);
	\draw (x5) to [bend right = 30] node[anchor=north]{$F_{k+2}$}(x6);
\end{tikzpicture}
\caption{A 4-cycle in $G_n$}\label{fi:last4cycle}
\end{figure}
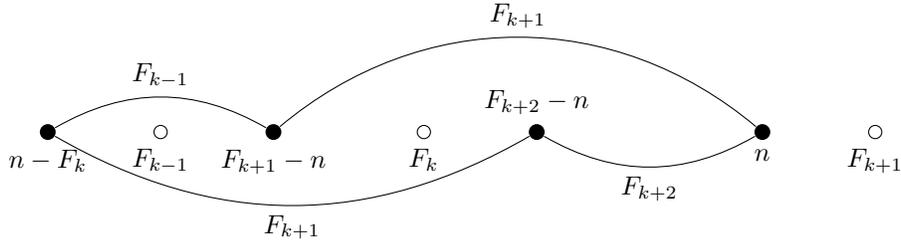

\begin{theorem}
Let $k \geq 2$ and $n=F_{2k+3}-1$. Then $G_n$ contains a cycle of length $2k$.
\end{theorem}
\begin{proof}
Define $c_1=1$, $c_2=4$,
 and for each $i=3,\ldots,2k$, recursively define $c_i=F_{i+4}-c_{i-1}$.
For each $i\geq 2$, $c_{i-1}+c_i$ is a Fibonacci number and so $\{c_{i-1},c_i\}$ is an edge
in $G_n$. To show that
$(c_1,c_2,...,c_{2k})$ is a cycle of length $2k$ in $G_n$, it remains to show that
$c_{2k}=F_{2k+3}-1$, in which case $1+c_{2k}$ is also a Fibonacci number.
When $k=2$, $(c_1,c_2,c_3,c_4)=(1,4,9, 12)$ is a cycle, so assume that $k\geq 3$.
Calculating,
 \begin{align*}
 c_{2k}&=F_{2k+4}-c_{2k-1}\\
           &=F_{2k+4}-(F_{2k+3}-c_{2k-2})\\
           &=F_{2k+2}+c_{2k-2}\\
           &=F_{2k+2}+F_{2k+2}-c_{2k-3}\\
           &=F_{2k+2}+F_{2k+2}-(F_{2k+1}-c_{2k-4})\\
           &=F_{2k+2}+F_{2k}+c_{2k-4}\\
           &\ \ \vdots \\
           &=\left(\sum_{i=3}^{k+1}F_{2i}\right)+c_2\\
           & =\left(\sum_{i=0}^{k+1}F_{2i}\right)-(F_0+F_2+F_4)+4\\
           &=\sum_{i=0}^{k+1}F_{2i}\\
           &=F_{2k+3}-1,&&\mbox{(by Fact \ref{fa:sumevenfibs})}
\end{align*}
as desired.
\end{proof}

\begin{theorem} \label{th:nocross}
Let $n\geq 7$.
 If $C = (a_1,a_2,...,a_m)$ is a cycle in the Fibonacci-sum graph $G_n$, 
 then there do not exist
 edges $\{a_i,a_k\}$ and $\{a_j,a_\ell\}$  in $C$ with $i < j < k < \ell$; in other words,
   there are no crossing chords inside $C$.
 \end{theorem}
 \begin{proof} The proof is by induction on $n$. 
 
 When $n=7$, there is only one cycle, which has no chords and hence
  no crossing chords.  Let $N\geq 7$, and assume that for all
 $n\leq N$, the statement holds.  It remains to show that  $G_{N+1}$ has no crossing
 chords.

 Let $k$ be so that $F_k\leq N+1 <F_{k+1}$ and let $C = (a_1,a_2,...,a_m)$ be a cycle
 in $G_{N+1}$. If $N+1=F_k$, then $\deg(N+1)=1$, and so $N+1$ is not on $C$ and so by
 induction hypothesis, $G_{N+1}$ has no crossing chords.

 So assume that $F_k<N+1<F_{k+1}$. Suppose that $N+1$ is on $C$, since otherwise the 
 statement follows from induction hypothesis. 
 By Lemma  \ref{le:lastvertex}, the only neighbours 
 of $N+1$ are  $F_{k+1}-(N+1)$  and $F_{k+2}-(N+1)$. No chord in $C$ contains $N+1$ 
 (since otherwise, the degree of $N+1$ would  be at least 3). 
 Similarly, $F_{k+2}-(N+1)$ has only two neighbours in $G$
 (namely $N+1-F_k$ and $N+1$), and so is not the endpoint of a chord.
 Thus, for some $b\in\{1,\ldots,m\}$, let $a_{b}=N+1-F_k$, $a_{b+1}=F_{k+2}-(N-1)$,
 $a_{b+2}=N+1$, and $a_{b+3}=F_{k+1}-(N+1)$.

 Then
  $C=(a_1, a_2, ..., a_{b-1}, N+1-F_k, F_{k+2}-(N+1), N+1, F_{k+1}-(N+1), a_{b+4},\ldots, a_m)$.
 Consider
 $C^{\prime}=(a_1, a_2, ..., a_{b-1}, N+1-F_k,  F_{k+1}-(N+1), a_{b+4}, \ldots, a_m)$.
  Since $C^\prime$ is contained in $G_N$, by the induction hypothesis,
   $C^{\prime}$ has no crossing chords. Therefore, $C$ has no crossing chords.
 \end{proof}
 
 Although any cycle in $G_n$ does not contain crossing chords,  $G_n$ satisfies a property
 similar to being chordal (where every cycle induces a chord). First note that any copy of
a 4-cycle in $G_n$ has no chords (since a chord would then create a triangle, 
which is impossible  because $G_n$ is bipartite).

 \begin{theorem}\label{th:longcyclechord}
 For any $n \geq 1$, every cycle of $G_n$  of length at least 6 contains a 
 chord that has exactly two vertices of the cycle on one side of the chord 
 (so forming a 4-cycle).
 \end{theorem}
 \begin{proof}
 Fix $n$  and let $C$ be a cycle of length at least 6 in $G_n$.
  Let $m$ be the vertex in  $C$  with the largest value, and let $k$ be so that
  $F_k \leq m \leq F_k+1$.
  Then vertices $F_{k+1}-m$ and $F_{k+2}-m$ are the only
  neighbours of $m$ in $G_m$, and so are both adjacent to $m$ in $C$.
  However, $F_{k+2}-m$ is
  adjacent to $m-F_k$ in $G_m$, and therefore $\{F_{k+1}-m,m-F_k\}$ is a 
  desired chord in $C$.
 \end{proof}

 \section{Treewidth}\label{se:treewidth}
 
 The ``treewidth'' of a graph $G$ is a measure of how close $G$ is to 
 being a tree.
 The notions of tree decompositions and 
 treewidth (under the name ``dimension'') were introduced by Halin \cite{Hali:76}
 in 1976, and independently, later by Seymour and Thomas (see \cite{SeTh:86}) in
 in their study of graph minors. 
  \begin{definition}
  A {\em tree decomposition} of a graph 
 $G$ is a pair $(T,\mathcal{ V})$, where $T$ is a tree and 
 $\mathcal{ V}=\{V_t\subseteq V(G):t\in V(T)\}$ is a family
 of subsets of $V(G)$ indexed by vertices in $T$ so that three conditions hold:
(i) $\cup_{i\in V(T)}=V$;
 (ii)  for every edge $\{x,y\}\in E(G)$, there exists $V_t\in \mathcal{ V}$
 containing both $x$ and $y$;
 (iii) if $t_it_jt_k$ is a path in $T$, then $V_i\cap V_k\subseteq V_j$.
 The width of a tree decomposition is $(T,\mathcal{V})$ is $\max\{|V_t|-1:t\in V(T)\}$
 and the {\em treewidth} of $G$, denoted tw$(T)$
  is the minimum width of any tree-decomposition of $G$.
 \end{definition}
 
 There are different characterizations of treewidth, some of which are
 described in \cite{Dies:10}.  For example, 
  Seymour and Thomas  \cite{SeTh:93} characterized treewidth
  in terms of brambles.
Another  characterization of treewidth uses chordal graphs and clique numbers.
For a graph $H$, let $\omega(H)$ denote the clique number of $H$, 
the order of the largest complete subgraph of $H$.
\begin{lemma}[see {\cite[Cor.~12.3.12]{Dies:10}}]\label{le:twchordal}
For any graph $G$,  
\[ \mbox{tw}(G)=\min\{\omega(H)-1:G\subseteq H, H \mbox{ chordal}\}.\]
   \end{lemma}
So the treewidth of a tree is 1, the treewidth of a cycle is 2, 
and the treewidth of the complete graph $K_n$ is $n-1$.

For $n=1,\ldots, 6$, the treewidth of $G_n$  is equal to 1 since these graphs are trees. 
However, for $n\geq 7$  each $G_n$ contains a cycle and so has treewidth at least 2.
 \begin{theorem}\label{th:treewidth}
 For each $n\geq 7$, $G_n$ has treewidth 2.
 \end{theorem}
\begin{proof} The proof given here is by induction 
 and uses Lemma \ref{le:twchordal}.
A second proof is given in Section \ref{se:planarity} as a Corollary to 
Theorem \ref{th:outerplanar} on outerplanarity. \medskip

For the base step, by inspection, $G_7$ has treewidth 2.
\medskip

For the inductive step, let  $m \geq 8$ and assume that $G_{m-1}$ has treewidth 2. 
Then by
definition, there is a chordal graph $H_{m-1}$ on $m$ vertices that contains 
$G_{m-1}$ and has clique number 3.  It remains to
show  that for $G_{m}$ there is a chordal graph $H_{m}$ on $m$ vertices 
that contains $G_{m}$ and 
has clique number 3. Assume that $F_{k}\leq m < F_{k+1}$. 
By Lemma \ref{le:lastvertex}, the degree of $m$ in $G_m$ is either 1 or 2. 

If   $\deg_{G_m}(m)=1$, then $H_{m}$ can be taken to be $H_{m-1}$ plus 
the edge $\{m, F_{k+1}-m\}$.
So assume that $\deg_{G_m}(m)=2$. Then $m$ is adjacent to $F_{k+2}-m$ 
and $F_{k+1}-m$, 
both of which are adjacent to $m-F_{k}$. Construct $H_{m}$ by 
adding vertex $m$ to $H_{m-1}$ 
and the edges $\{m, F_{k+2}-m\}$, $\{m, F_{k+1}-m\}$ and $\{m, m-F_k\}$. 
Then $H_m$ contains $G_{m}$ as a subgraph and $\deg_{H_m}(m)=3$.

Any clique with 3 or more vertices in $H_m$ that contains vertex $m$
 does not contain both of 
the vertices $F_{k+2}-m$ and $F_{k+1}-m$ since they are not adjacent. 
Therefore any clique in $H_m$ that contains $m$ has  at most 3 vertices. 
By the induction hypothesis, any clique in $H_m$ that does not contain $m$
 has at most 3 vertices. 
Therefore, any clique in $H_m$ has at most 3 vertices.
This completes the proof of induction step and hence the theorem.
\end{proof}

\section{Planarity}\label{se:planarity}
The following theorem is known (see, e.g., \cite[p.27]{Cost:08} for a suggested
proof outline);
however, a proof is included here for completeness. 
\begin{theorem} \label{th:planar}
For $n \geq 1$, $G_n$ is planar.
\end{theorem}
\begin{proof}
The proof is by induction on $n$.
For $n\leq 6$, $G_n$ is planar (see Figure \ref{fi:G6}).   Fix  $m\geq 6$
and assume that $G_m$ is planar (with a fixed plane drawing).  It remains to
show  that   $G_{m+1}$ s planar. Let $\ell$ be the positive integer
such that $F_{\ell} \leq m+1 < F_{\ell+1}$. Consider two possible scenarios.
\medskip

Case 1: Suppose $m+1 \leq \frac{1}{2}F_{\ell+2}$. By Lemma \ref{le:lastvertex},
the vertex $v=m+1$ has degree 1 in $G_{m+1}$ ,
in which case it follows that $G_{m+1}$ is planar, since adding any pendant
edge to a planar drawing of a graph can be done without crossing any edges.
\medskip

Case 2: Suppose $m+1> F_{\ell+2}$. Then in $G_{m+1}$, the vertex 
$v=m+1$ is adjacent to
vertices  $v_1 = F_{\ell+1} - (m+1)$ and $v_2 = F_{\ell+2} - (m+1)$.
The vertex $v_3=m+1-F_\ell$ is adjacent to both $v_1$ and $v_2$ in 
$G_{m+1}$ and in $G_m$.
The vertex $v_2$ has degree 1 in $G_m$. Let $F$ be one of the faces with $v_1v_3$ as
a bordering edge, and, if necessary, move $v_2$ into $F$. Then $v$ can also be
placed in $F$ (so that the edges $vv_1$ and $vv_2$ do not cross any other edges;
 {\it e.g.,} put $v$ inside the triangular region formed by $v_1v_3v_2$),
giving a planar drawing  $G_{m+1}$. This concludes the inductive step, and hence the proof.
\end{proof}

Below in Theorem \ref{th:outerplanar},
it is shown that each $G_n$ is outerplanar;
the following lemma is used for the proof of this fact.

\begin{lemma} \label{le:nobooksubgraph}
Let $H$ be the graph on eight vertices consisting of  three 4-cycles sharing a single edge. 
For $n \geq 1$, $G_n$ does not contain a subgraph isomorphic to $H$.
\end{lemma}
\begin{proof}
The proof is by induction on $n$.\medskip

For  $n \leq 7$, the statement of the theorem is trivially true  because $H$ has eight vertices.
\smallskip

Fix  $m \geq 8$ and assume that $G_{m-1}$ does not contain $H$.  
To complete the induction step, 
it remains to show that  $G_{m}$ does not contain $H$. 
Let the integer $k$ be so that $F_{k}\leq m < F_{k+1}$.

In hopes of a contradiction, suppose that $G_{m}$ contains a copy of $H$.  
By the induction
 hypothesis, $m$ is a vertex of $H$. Since vertex $m$ belongs to $H$, 
 the degree of $m$ is 2. 
 Therefore, $m>\frac{F_{k+2}}{2}$. 
 The only neighbours of $m$ in $G_{m}$ are $F_{k+1}-m$ and $F_{k+2}-m$. 
 Vertex $F_{k+2}-m>F_{k+2}-F_{k+1}=F_k$ is also a vertex of degree 2. 
 The other neighbour of $F_{k+2}-m$ is $m-F_{k}$. 
 Therefore, vertices $F_{k+2}-m$, $F_{k+1}-m$, $m$ and $m-F_{k}$ belong to $H$.
 
Vertex $m-F_{k}$ has degree at least 4 in $G_{m}$, since $m-F_k$  belongs to $H$. 
That may only happen if $m-F_{k}<F_{k-2}$, because $m-F_{k}$ added to other vertex 
has to make $F_{k-2}$, $F_{k-1}$, $F_{k}$ and $F_{k+1}$. 
Therefore, $\frac{F_{k+2}}{2}<m<F_{k}+F_{k-2}$; in particular $F_{k}+F_{k-3}<m$. 
Vertex $m-F_{k}$ has degree precisely 4 in $G_m$ and is adjacent to $F_{k+2}-m$, 
$F_{k+1}-m$, $F_{k}+F_{k-2}-m$ and $2F_{k}-m$. Then,
\[F_{k}>2F_{k}-(F_{k}+F{k-3})>2F_{k}-m>2F_k-(F_k+F_{k-2})=F_{k-1}.\]

Vertex $2F_{k}-m$ has degree 2 and is adjacent to $m-F_{k}$ and $m-F_{k-2}$. 
Vertex $m-F_{k-2}$ is not adjacent to vertex $F_{k+1}-m$, and therefore 
vertex $m$ is not in a copy of $H$ in $G_{m}$. 
By the inductive hypothesis, there is no copy of $H$ in $G_{m-1}$,
 and so $G_{m}$ does not contain $H$ as a subgraph. 
 This ends the proof of induction step and a proof of the lemma.
\end{proof}

\begin{theorem} \label{th:outerplanar}
For $n \geq 1$, $G_n$ is outerplanar.
\end{theorem}
\begin{proof} The proof is by induction on $n$.
For values $n\leq 7$, the statement of the theorem is true by inspection.

Let  $m \geq 8$ and assume that $G_{m-1}$ is outerplanar 
(with a fixed outerplanar drawing).  It remains to
show  that   $G_{m}$ is outerplanar. Let $k$ be so that $F_{k}\leq m < F_{k+1}$.

If $\deg_{G_m}(x)=1$, then drawing the edge $\{m, F_{k+1}-m\}$ in the outer
 face of  the outerplanar drawing of $G_{m-1}$ produces an outerplanar 
 drawing of $G_{m}$.

If $\deg_{G_m}(x)=2$, then $m$ is adjacent to vertices $F_{k+2}-m$ and $F_{k+1}-m$. 
Vertex $F_{k+2}-m$ has degree 2 in $G_{m}$ and is also adjacent to vertex $m-F_{k}$. Vertex $m-F_{k}$ is adjacent to $F_{k+1}-m$.\medskip

\noindent{\sc Claim}: The edge $\{F_{k+1}-m, m-F_{k}\}$ borders the outer face of
 the outerplanar drawing of $G_{m-1}$.\\
 \noindent{\sc Proof of claim}: Suppose that this is not the case and derive a contradiction. 
 So assume $\{F_{k+1}-m, m-F_{k}\}$ borders two inner faces of the
 drawing of $G_{m-1}$. These two inner faces of the outerplanar drawing of
  $G_{m-1}$ do not have two edges in common. 
  Then by Theorem \ref{th:longcyclechord} there are two cycles 
  $\{F_{k+1}-m, m-F_{k}, x , y\}$ and $\{F_{k+1}-m, m-F_{k}, z, t\}$ with 
  all four vertices $x,y,z,t$ being different.
Then vertices $\{m, F_{k+2}-m, F_{k+1}-m, m-F_{k}, x, y, z, t\}$ induce a copy
 of $H$ in $G_m$ which contradicts Lemma \ref{le:nobooksubgraph}, proving
 the claim.\medskip
 
 To complete the proof, by the claim, 
 put vertices $m$ and $F_{k+2}-m$ in the outer face next 
 to $\{F_{k+1}-m, m-F_{k}\}$ in the drawing of $G_{m-1}$ to form 
 an outerplanar drawing of $G_m$.
This completes the proof of the induction step and hence the theorem.
\end{proof}

As promised in Section \ref{se:treewidth}, Theorem \ref{th:treewidth}
is now also a corollary to Theorem \ref{th:outerplanar} because
of the well-known result (see, e.g., \cite{Bodl:88},  \cite{Bodl:96},
or \cite{CoWa:83}) that  outerplanar graphs have treewidth at most 2.

\section{Automorphisms of $G_n$}\label{se:automorphisms}

Any automorphism of $G_n$ 
 takes endpoints of a Hamiltonian path to endpoints of another
Hamiltonian path.  By Theorem \ref{th:FKMOPmain}, Fibonacci-sum graphs have
either no, one, or two Hamiltonian paths  so one might expect that in most
cases, there are few automorphisms.

 For $n\leq 2278$, calculations show that 
 the $G_n$s with a trivial automorphism group are
those where $n$ is in the intervals $[7,10]$, $[17,21]$, $[30,50]$, $[72,92]$, 
$[127,215]$, $[305,393]$, $[538,914]$, $[1292,1668]$.
The rest have order two.

All but the first interval for each type have lengths that are Fibonacci numbers.
  For those with just one automorphism, lengths are
4, 5, 21, 21, 89, 89, 377, 377, which are, respectively,
 4, $F_5$, $F_8$, $F_8$, $F_{11}$, $F_{11}$, ...;
similarly, for order 2 the intervals have lengths, 6, 8, 21, 34, 89, 123, 
all but the first being Fibonacci numbers.

Denote the automorphism group of a graph $G$ by $\mbox{Aut}(G)$.
In the following theorem, ``id'' denotes an identity map. In the proof of the
next theorem,  standard functional notation is used. For any set $V$
and any function $f$ with domain $X$, write $f(X)=\{f(x):x\in X\}$.  Say that
$f:X\rightarrow X$ ``fixes''   $x\in X$ iff $f(x)=x$.  For any 
subset $W\subseteq X$ say that ``$f$ fixes $W$'' iff $f(W)=W$.
Note that a function  can fix a set without necessarily fixing every
vertex in that set. 

\begin{theorem}\label{th:Automorphism}
Let $n\geq 9$ be an integer that is divisible by 3.\smallskip

If $N \in \left[F_n, \frac{3}{2} F_n \right)\cup \left[\frac{1}{2}F_{n+3}, F_{n+2}+\frac{1}{2}F_{n-3} \right)\cup \left[F_{n+3}-\frac{1}{2}F_n, F_{n+3} \right)$, 
then Aut$(G_N)=\{id\}$.
\smallskip

If $N \in \left[\frac{3}{2}F_{n}, \frac{1}{2}F_{n+3}\right)$, then 
Aut$(G_N)=\{id, \phi\},$ where $\phi$ interchanges 
$\frac{1}{2}F_n$ with $\frac{3}{2}F_n$ and fixes all other vertices.\smallskip

If $N \in \left[F_{n+2}+\frac{1}{2}F_{n-3}, F_{n+3}-\frac{1}{2}F_{n}\right)$, then 
Aut$(G_N)=\{id, \phi\}$, where $\phi$  interchanges 
$\frac{1}{2}F_n$ with $\frac{3}{2}F_n$, interchanges $F_{n+2}+\frac{1}{2}F_{n-3}$ with 
$\frac{1}{2}F_{n+3}$,  and fixes all other vertices.
\end{theorem}

\begin{proof}
The proof is  by induction on $n$. For each integer $n\geq 9$ (divisible by 3),
 let $A(n)$ be the statement of the theorem. \medskip

\noindent{\sc Base step}: For each $i\in [34,144)=[F_9, F_{12})$
the statement $A(i)$  is  verified by direct calculation.\medskip

\noindent{\sc Induction step};  Let $k\geq 12$ satisfy $3 \mid k$, and assume that  
for every $m\in [F_9, F_k)$, the statement $A(m)$ is true.
It remains to show that for all $N\in \left[F_k ,F_{k+3}\right)$, the statement $A(N)$
is true.
The proof consists of five cases, each case corresponding to one of
the five half-open intervals  given in the statement of the theorem.  All but the last
case have subcases.  For each case, a diagram indicating degrees in certain
intervals is given (not to scale) spanning two lines because of space
limitations.

Throughout the proof, let $\phi:[N]\rightarrow [N]$ denote an automorphism
of $G_N$. 
 \bigskip

\noindent{\sc Case} 1: $N \in \left[F_k, \frac{1}{2}F_{k+2}\right)$.
There are three subcases to consider.\medskip

{\sc Case} 1a: $N \in \left[F_k, F_k+\frac{1}{2}F_{k-3}\right)$:

\begin{center}
\begin{tikzpicture}[scale=.85]
\draw[black, thick] (-6.4,0) -- (6.4,0);
\filldraw[black] (-6.4,0) circle (1pt) node[anchor=south] {$F_{k-5}$};
\filldraw[black] (-3.2,0) circle (1pt) node[anchor=south] {$F_{k-4}$};
\filldraw[black] (0,0) circle (1pt) node[anchor=south] {$F_{k-3}$};
\filldraw[black] (3.2,0) circle (1pt) node[anchor=south] {$F_{k-2}$};
\filldraw[black] (6.4,0) circle (1pt) node[anchor=south] {$F_{k-1}$};
\filldraw[black] (-4.8,0) circle (1pt) node[anchor=south] {$\frac{F_{k-3}}{2}$};
\filldraw[black] (4.8,0) circle (1pt) node[anchor=south] {$\frac{F_{k}}{2}$};

\draw[black, thick] (-6.4,-2) -- (6.4,-2);
\filldraw[black] (-6.4,-2) circle (1pt) node[anchor=south] {$F_{k-1}$};
\filldraw[black] (-3.2,-2) circle (1pt) node[anchor=south] {$F_{k}$};
\filldraw[black] (0,-2) circle (1pt) node[anchor=south] {$F_{k+1}$};
\filldraw[black] (3.2,-2) circle (1pt) node[anchor=south] {$F_{k+2}$};
\filldraw[black] (6.4,-2) circle (1pt) node[anchor=south] {$F_{k+3}$};
\filldraw[black] (1.6,-2) circle (1pt) node[anchor=south] {$\frac{F_{k+3}}{2}$};
\filldraw[black] (-2.2,-2) circle (1pt) node[anchor=north] {1};
\filldraw[red] (-2.2,-2) circle (0pt) node[anchor=south] {$N$};

\foreach \x in {1,...,5}
\filldraw[black] (-2.2-0.2*\x,-2) circle (0.5pt) node[anchor=north] {1};

\foreach \x in {1,...,16}
\filldraw[black] (-3.2-0.2*\x,-2) circle (0.5pt) node[anchor=north] {2};

\filldraw[black] (6.4,0) circle (0.5pt) node[anchor=north] {2};

\foreach \x in {1,...,5}
\filldraw[black] (6.4-0.2*\x,0) circle (0.5pt) node[anchor=north] {3};

\filldraw[red] (5.4,1) circle (0pt) node[anchor=south] {{$F_{k+1}-N$}};
\draw[dotted, thick] (5.4,1) -- (5.4,0);

\foreach \x in {1,...,2}
\filldraw[black] (5.4-0.2*\x,0) circle (0.5pt) node[anchor=north] {2};

\filldraw[black] (4.8,0) circle (0.5pt) node[anchor=north] {1};

\foreach \x in {1,...,8}
\filldraw[black] (4.8-0.2*\x,0) circle (0.5pt) node[anchor=north] {2};

\foreach \x in {1,...,16}
\filldraw[black] (3.2-0.2*\x,0) circle (0.5pt) node[anchor=north] {3};

\foreach \x in {1,...,16}
\filldraw[black] (0-0.2*\x,0) circle (0.5pt) node[anchor=north] {4};

\foreach \x in {1,...,7}
\filldraw[black] (-3.2-0.2*\x,0) circle (0.5pt) node[anchor=north] {5};

\filldraw[black] (-4.8,0) circle (0.5pt) node[anchor=north] {4};

\foreach \x in {1,...,8}
\filldraw[black] (-4.8-0.2*\x,0) circle (0.5pt) node[anchor=north] {5};

\end{tikzpicture}
\end{center}
\medskip


 Let $S=\{1,...,F_k-1\}$. 
 Then $S$ is a set of vertices of degree at least 2  together with the
  vertex $\frac{1}{2}F_k$.
 
The vertex $\frac{1}{2}F_k$ is the only degree 1 vertex 
adjacent to a degree 4 vertex (namely, $\frac{1}{2}F_{k-3}$), and so
 $\phi$ fixes $\frac{1}{2}F_k$. Since any automorphism sends
  vertices of degree at least 2 vertices of degree at least 2, 
  $\phi$ sends all remaining vertices of $S$ into $S$ and so  $\phi(S)=S$.
   The map  $\psi=\phi |_S $ is an automorphism of $G_{F_k-1}$,
 and so by the induction hypothesis, $\psi$ is the identity on $S$.

Now consider any $y \geq F_k$. The degree of $y$ is 1 and $y$ is adjacent to 
$F_{k+1}-y$. Since $F_{k+1}-y\leq F_{k-1}$,  $\phi(F_{k+1}-y)=F_{k+1}-y$.
 Vertex $y$ is adjacent to $F_{k+1}-y$, therefore $\phi(y)$ is adjacent to $F_{k+1}-y$.
 Since $y$ is the only degree 1 vertex adjacent to $F_{k+1}-y$, it 
 follows that  $\phi(y)=y$.

Therefore,  the only automorphism of $G_N$ is the identity.\bigskip

{\sc Case} 1b: $N \in \left[ F_k+\frac{1}{2}F_{k-3}, \frac{1}{2}F_{k+2}\right)$:
Note that $\frac{1}{2}F_{k+2}$ is not an integer.\medskip

\begin{center}
\begin{tikzpicture}[scale=.85]
\draw[black, thick] (-6.4,0) -- (6.4,0);
\filldraw[black] (-6.4,0) circle (1pt) node[anchor=south] {$F_{k-5}$};
\filldraw[black] (-3.2,0) circle (1pt) node[anchor=south] {$F_{k-4}$};
\filldraw[black] (0,0) circle (1pt) node[anchor=south] {$F_{k-3}$};
\filldraw[black] (3.2,0) circle (1pt) node[anchor=south] {$F_{k-2}$};
\filldraw[black] (6.4,0) circle (1pt) node[anchor=south] {$F_{k-1}$};
\filldraw[black] (-4.8,0) circle (1pt) node[anchor=south] {$\frac{F_{k-3}}{2}$};
\filldraw[black] (4.8,0) circle (1pt) node[anchor=south] {$\frac{F_{k}}{2}$};

\draw[black, thick] (-6.4,-2) -- (6.4,-2);
\filldraw[black] (-6.4,-2) circle (1pt) node[anchor=south] {$F_{k-1}$};
\filldraw[black] (-3.2,-2) circle (1pt) node[anchor=south] {$F_{k}$};
\filldraw[black] (0,-2) circle (1pt) node[anchor=south] {$F_{k+1}$};
\filldraw[black] (3.2,-2) circle (1pt) node[anchor=south] {$F_{k+2}$};
\filldraw[black] (6.4,-2) circle (1pt) node[anchor=south] {$F_{k+3}$};
\filldraw[black] (1.6,-2) circle (1pt) node[anchor=south] {$\frac{F_{k+3}}{2}$};

\filldraw[black] (-1.8,-2) circle (1pt) node[anchor=north] {1};
\filldraw[red] (-1.8,-2) circle (0pt) node[anchor=south] {$N$};

\foreach \x in {1,...,7}
\filldraw[black] (-1.8-0.2*\x,-2) circle (0.5pt) node[anchor=north] {1};

\foreach \x in {1,...,16}
\filldraw[black] (-3.2-0.2*\x,-2) circle (0.5pt) node[anchor=north] {2};

\filldraw[black] (6.4,0) circle (0.5pt) node[anchor=north] {2};

\foreach \x in {1,...,7}
\filldraw[black] (6.4-0.2*\x,0) circle (0.5pt) node[anchor=north] {3};

\filldraw[black] (4.8,0) circle (0.5pt) node[anchor=north] {2};

\filldraw[red] (4.2,1) circle (0pt) node[anchor=south] {{$F_{k+1}-N$}};
\draw[dotted, thick] (4.2,1) -- (4.2,0);

\foreach \x in {1,...,2}
\filldraw[black] (4.8-0.2*\x,0) circle (0.5pt) node[anchor=north] {3};

\filldraw[black] (4.2,0) circle (0.5pt) node[anchor=north] {3};

\foreach \x in {1,...,5}
\filldraw[black] (4.2-0.2*\x,0) circle (0.5pt) node[anchor=north] {2};

\foreach \x in {1,...,16}
\filldraw[black] (3.2-0.2*\x,0) circle (0.5pt) node[anchor=north] {3};

\foreach \x in {1,...,16}
\filldraw[black] (0-0.2*\x,0) circle (0.5pt) node[anchor=north] {4};

\foreach \x in {1,...,7}
\filldraw[black] (-3.2-0.2*\x,0) circle (0.5pt) node[anchor=north] {5};

\filldraw[black] (-4.8,0) circle (0.5pt) node[anchor=north] {4};

\foreach \x in {1,...,8}
\filldraw[black] (-4.8-0.2*\x,0) circle (0.5pt) node[anchor=north] {5};
\end{tikzpicture}
\end{center}
\medskip

The proof is nearly the same as for Case 1a, with the exception that
now the degree of  $\frac{1}{2}F_k$ is 2.
 Let $S=\{1,...,F_k-1\}$. 
Since any automorphism maps vertices of degree at least 2
  to vertices of degree at least 2,  $\phi(S)=S$, and the remainder
  of the proof is identical.\bigskip

{\sc Case} 1c: $N \in \left(\frac{1}{2}F_{k+2},\frac{3}{2}F_k\right)$.

\begin{center}
\begin{tikzpicture}[scale=.85]
\draw[black, thick] (-6.4,0) -- (6.4,0);
\filldraw[black] (-6.4,0) circle (1pt) node[anchor=south] {$F_{k-5}$};
\filldraw[black] (-3.2,0) circle (1pt) node[anchor=south] {$F_{k-4}$};
\filldraw[black] (0,0) circle (1pt) node[anchor=south] {$F_{k-3}$};
\filldraw[black] (3.2,0) circle (1pt) node[anchor=south] {$F_{k-2}$};
\filldraw[black] (6.4,0) circle (1pt) node[anchor=south] {$F_{k-1}$};
\filldraw[black] (-4.8,0) circle (1pt) node[anchor=south] {$\frac{F_{k-3}}{2}$};
\filldraw[black] (4.8,0) circle (1pt) node[anchor=south] {$\frac{F_{k}}{2}$};

\draw[black, thick] (-6.4,-2) -- (6.4,-2);
\filldraw[black] (-6.4,-2) circle (1pt) node[anchor=south] {$F_{k-1}$};
\filldraw[black] (-3.2,-2) circle (1pt) node[anchor=south] {$F_{k}$};
\filldraw[black] (0,-2) circle (1pt) node[anchor=south] {$F_{k+1}$};
\filldraw[black] (3.2,-2) circle (1pt) node[anchor=south] {$F_{k+2}$};
\filldraw[black] (6.4,-2) circle (1pt) node[anchor=south] {$F_{k+3}$};
\filldraw[black] (1.6,-2) circle (1pt) node[anchor=south] {$\frac{F_{k+3}}{2}$};

\filldraw[black] (-1,-2) circle (1pt) node[anchor=north] {2};
\filldraw[red] (-1,-2) circle (0pt) node[anchor=south] {$N$};

\foreach \x in {1,...,6}
\filldraw[black] (-1-0.2*\x,-2) circle (0.5pt) node[anchor=north] {2};

\filldraw[red] (-2.2,-1.5) circle (0pt) node[anchor=south] {$F_{k+2}-N$};
\draw[dotted, thick] (-2.2,-1.5) -- (-2.2,-2);

\foreach \x in {1,...,5}
\filldraw[black] (-2.2-0.2*\x,-2) circle (0.5pt) node[anchor=north] {1};

\foreach \x in {1,...,16}
\filldraw[black] (-3.2-0.2*\x,-2) circle (0.5pt) node[anchor=north] {2};

\filldraw[black] (6.4,0) circle (0.5pt) node[anchor=north] {2};

\foreach \x in {1,...,7}
\filldraw[black] (6.4-0.2*\x,0) circle (0.5pt) node[anchor=north] {3};

\filldraw[black] (4.8,0) circle (0.5pt) node[anchor=north] {2};

\foreach \x in {1,...,8}
\filldraw[black] (4.8-0.2*\x,0) circle (0.5pt) node[anchor=north] {3};

\foreach \x in {1,...,16}
\filldraw[black] (3.2-0.2*\x,0) circle (0.5pt) node[anchor=north] {4};

\foreach \x in {1,...,5}
\filldraw[black] (0-0.2*\x,0) circle (0.5pt) node[anchor=north] {5};

\filldraw[red] (-1,0.5) circle (0pt) node[anchor=south] {$F_{k+1}-N$};
\draw[dotted, thick] (-1,0.5) -- (-1,0);

\foreach \x in {1,...,11}
\filldraw[black] (-1-0.2*\x,0) circle (0.5pt) node[anchor=north] {4};

\foreach \x in {1,...,7}
\filldraw[black] (-3.2-0.2*\x,0) circle (0.5pt) node[anchor=north] {5};

\filldraw[black] (-4.8,0) circle (0.5pt) node[anchor=north] {4};

\foreach \x in {1,...,8}
\filldraw[black] (-4.8-0.2*\x,0) circle (0.5pt) node[anchor=north] {5};
\end{tikzpicture}
\end{center}
\medskip

Let $S=\{1,...,F_{k-1}-1\}$. Then $S$ is a set of vertices of degree at least 3 
together with the vertex $\frac{1}{2}F_k$, which is of degree 2. 

 The vertex $\frac{1}{2}F_k$ is the only vertex of degree 2  adjacent to a degree 4 vertex (namely, $\frac{1}{2}F_{k-3}$) and a degree 1 vertex (namely, $F_k+\frac{1}{2}F_{k-3}$).
 All other degree 2 vertices have either a degree 2 neighbour or a degree 3 neighbour. 
 Therefore $\phi(\frac{1}{2}F_k)=\frac{1}{2}F_k$
 
 The vertices of degree at least 3 are mapped to the vertices of degree at least 3. 
So $\phi(S)=S$ and $\psi=\phi \mid_S$ is an automorphism of $G_{F_{k-1}-1}$.
 By the induction hypothesis, $\psi$ is the identity on $S$.
 
 The vertex $F_{k-1}$ is the only degree 2 vertex adjacent to both a degree 3 and
  a degree 1 vertex, and so $\phi(F_{k-1})=F_{k-1}$. So, for any $x \leq F_{k-1}$, 
  $\phi(x)=x$.
  
Now consider any $y \geq F_k$ with degree 1, and its only neighbour $F_{k+1}-y$.
Since $F_{k+1}-y\leq F_{k-1}$,  by the preceding paragraphs, $\phi(F_{k+1}-y)=F_{k+1}-y$. 
Since  $y$ and $F_{k+1}-y$ are adjacent,  $\phi(y)$ and  $F_{k+1}-y$ are
adjacent, but the only degree 1 vertex  adjacent to $F_{k+1}-y$ is $y$.
 Hence, $\phi(y)=y$. So $\phi$ fixes those $y\geq F_k$ with degree 1.

Define two types of degree 2 vertices: type one are those  in the interval
 $\left(F_{k-1}, F_{k}\right)$ and type two are in the interval  $\left[F_{k+2}-N, N\right]$.

First, type one vertices are not adjacent to type two vertices.

Let $x \in \left(F_{k-1},F_{k}\right)$ be a type one vertex. 
Then $x$ is adjacent to $F_{k+1}-x$---another type one vertex. 
Vertex $x$ is also adjacent to $F_{k}-x \leq F_{k-1}$. 
Vertex $F_{k+1}-x$ is adjacent to $x-F_{k-1} \leq F_{k-1}$. 
The sum of indexes of the vertices $F_{k+1}-x$ and $x-F_{k-1}$ is $F_k$.

Let $y \in \left[F_{k+2}-N,N\right]$ be a type two vertex. 
Then $y$ is adjacent to $F_{k+2}-y$---another type two vertex.
 Vertex $y$ is also adjacent to $F_{k+1}-y \leq F_{k-1}$. 
 Vertex $F_{k+2}-x$ is adjacent to $y-F_{k} \leq F_{k-1}$
 since $(F_{k+2}-y)+(y-F_{k})=F_{k+1}$.

Type one vertices are mapped to type one vertices, since every vertex adjacent to
two type one vertices has  neighbours that sum to $F_k$. 
Type two vertices are mapped to type two vertices, since every vertex adjacent to
 two type two vertices has neighbours that sum to $F_{k+1}$. 
 This means that $\phi$ preserves the type of degree 2 vertices.

For $x \in \left(F_{k-1},F_{k}\right)$, the vertex $F_k-x$ is adjacent to $x$.
 Therefore, $\phi(x)$ is adjacent to $\phi(F_k-x)=F_k-x$. 
 However, $x$ is the only type one vertex adjacent to $F_k-x$, and so $\phi(x)=x$.

For $y \in \left[F_{k+2}-N,N \right]$, the vertex $F_{k+1}-y$ is adjacent to $y$.
 Therefore, $\phi(y)$ is adjacent to $\phi(F_{k+1}-y)=F_{k+1}-y$. 
 However, $y$ is the only type two vertex that is adjacent to $F_{k+1}-y$, and  
 so $\phi(y)=y$.

Therefore, the only automorphism of $G_N$ is the identity. \bigskip

\noindent{\sc Case 2:} $N\in \left[\frac{3}{2}F_{k}, \frac{1}{2}F_{k+3} \right)$.
There are two subcases to consider.\medskip

{\sc Case} 2a: $N \in \left[\frac{3}{2}F_{k},F_{k+1}\right)$. 

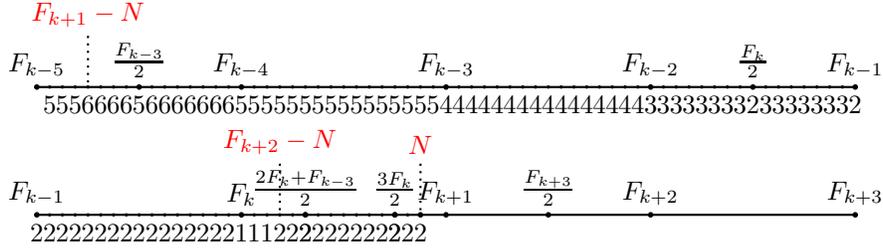
\begin{center}
\begin{tikzpicture}[scale=.85]
\draw[black, thick] (-6.4,0) -- (6.4,0);
\filldraw[black] (-6.4,0) circle (1pt) node[anchor=south] {$F_{k-5}$};
\filldraw[black] (-3.2,0) circle (1pt) node[anchor=south] {$F_{k-4}$};
\filldraw[black] (0,0) circle (1pt) node[anchor=south] {$F_{k-3}$};
\filldraw[black] (3.2,0) circle (1pt) node[anchor=south] {$F_{k-2}$};
\filldraw[black] (6.4,0) circle (1pt) node[anchor=south] {$F_{k-1}$};
\filldraw[black] (-4.8,0) circle (1pt) node[anchor=south] {$\frac{F_{k-3}}{2}$};
\filldraw[black] (4.8,0) circle (1pt) node[anchor=south] {$\frac{F_{k}}{2}$};

\draw[black, thick] (-6.4,-2) -- (6.4,-2);
\filldraw[black] (-6.4,-2) circle (1pt) node[anchor=south] {$F_{k-1}$};
\filldraw[black] (-3.2,-2) circle (1pt) node[anchor=south] {$F_{k}$};
\filldraw[black] (0,-2) circle (1pt) node[anchor=south] {$F_{k+1}$};
\filldraw[black] (3.2,-2) circle (1pt) node[anchor=south] {$F_{k+2}$};
\filldraw[black] (6.4,-2) circle (1pt) node[anchor=south] {$F_{k+3}$};
\filldraw[black] (1.6,-2) circle (1pt) node[anchor=south] {$\frac{F_{k+3}}{2}$};

\filldraw[black] (-0.4,-2) circle (1pt) node[anchor=north] {2};
\filldraw[red] (-0.4,-1.2) circle (0pt) node[anchor=south] {$N$};
\draw[dotted, thick] (-0.4,-1.2)--(-0.4,-2);

\filldraw[black] (-0.8,-2) circle (1pt) node[anchor=north] {2};
\filldraw[black] (-0.8,-2) circle (1pt) node[anchor=south] {$\frac{3F_k}{2}$};
\foreach \x in {1,...,11}
\filldraw[black] (-0.4-0.2*\x,-2) circle (0.5pt) node[anchor=north] {2};

\filldraw[red] (-2.6,-1.2) circle (0pt) node[anchor=south] {$F_{k+2}-N$};
\draw[dotted, thick] (-2.6,-1.2) -- (-2.6,-2);

\filldraw[black] (-2.2,-2) circle (1pt) node[anchor=north] {2};
\filldraw[black] (-2.2,-2) circle (1pt) node[anchor=south] {$\frac{2F_k+F_{k-3}}{2}$};

\foreach \x in {1,...,3}
\filldraw[black] (-2.6-0.2*\x,-2) circle (0.5pt) node[anchor=north] {1};

\foreach \x in {1,...,16}
\filldraw[black] (-3.2-0.2*\x,-2) circle (0.5pt) node[anchor=north] {2};

\filldraw[black] (6.4,0) circle (0.5pt) node[anchor=north] {2};

\foreach \x in {1,...,7}
\filldraw[black] (6.4-0.2*\x,0) circle (0.5pt) node[anchor=north] {3};

\filldraw[black] (4.8,0) circle (0.5pt) node[anchor=north] {2};

\foreach \x in {1,...,8}
\filldraw[black] (4.8-0.2*\x,0) circle (0.5pt) node[anchor=north] {3};

\foreach \x in {1,...,16}
\filldraw[black] (3.2-0.2*\x,0) circle (0.5pt) node[anchor=north] {4};

\foreach \x in {1,...,16}
\filldraw[black] (0-0.2*\x,0) circle (0.5pt) node[anchor=north] {5};

\foreach \x in {1,...,7}
\filldraw[black] (-3.2-0.2*\x,0) circle (0.5pt) node[anchor=north] {6};

\filldraw[black] (-4.8,0) circle (0.5pt) node[anchor=north] {5};

\foreach \x in {1,...,4}
\filldraw[black] (-4.8-0.2*\x,0) circle (0.5pt) node[anchor=north] {6};
\filldraw[red] (-5.6,0.8) circle (0pt) node[anchor=south] {$F_{k+1}-N$};
\draw[dotted, thick] (-5.6,0.8) -- (-5.6,0);

\foreach \x in {1,...,3}
\filldraw[black] (-5.6-0.2*\x,0) circle (0.5pt) node[anchor=north] {5};
\end{tikzpicture}
\end{center}
\medskip

Let $S=\{1,...,F_{k-2}-1\}$ and let $\phi\in \mbox{Aut}(G_N)$.
 Then $S$ is a set of vertices of degree at least 4 and since any automorphism
 maps  vertices of degree at least 4  to vertices of degree at least 4,
  $\phi(S)=S$.  The map $\psi=\phi \mid_S $ is then an automorphism 
 of $G_{F_{k-2}-1}$. By  the induction hypothesis, $\psi$ is either the identity 
  or interchanges 
 $\frac{3}{2}F_{k-3}$ with $\frac{1}{2}F_{k-3}$ and fixes all other vertices of $S$.
  However, in $G_N$, the degrees of $\frac{3}{2}F_{k-3}$ and $\frac{1}{2}F_{k-3}$ are 4 and 5 respectively, so no automorphism carries one to the other. Thus, $\psi$ is the identity map on $S$.

 Vertex $F_{k-2}$ is the only degree 3 vertex that does not have a 
 neighbour of degree 3, and so, $\phi(F_{k-2})=F_{k-2}$. 
 Each  remaining degree 3 vertex has a unique neighbour  less than $F_{k-3}$. 
 Therefore $\phi$ fixes degree 3 vertices. 
 So $\phi$ fixes every vertex in $[1,F_{k-1})\backslash\{\frac{1}{2}F_k\}$.
 (The last point $\frac{1}{2}F_k$ in this interval is discussed below.)

Vertex $F_k$ is the only degree 1 vertex that has a neighbour of degree 2, and so
 $\phi(F_k)=F_k$. 

Consider a degree 1 vertex $y > F_k$.
Then $y$ is adjacent to  $F_{k+1}-y$. 
Since $F_{k+1}-y\leq F_{k-1}$, by above,
 $\phi(F_{k+1}-y)=F_{k+1}-y$. Since $y$ is adjacent to $F_{k+1}-y$, 
  $\phi(y)$ is adjacent to $\phi(F_{k+1}-y)=F_{k+1}-y$. Since $y$ is the only
  degree 1 vertex  adjacent to $F_{k+1}-y$, it follows that $\phi(y)=y$.

Identify three types of degree 2 vertices: say those of
type one are in the interval $\left(F_{k-1}, F_k\right)$,
those of type two are in the interval 
$\left[F_{k+2}-N, N\right]\backslash \{F_k+F_{k-3}, \frac{3}{2}F_k\}$, and 
call the remaining three $\frac{1}{2}F_k$, $F_k+\frac{1}{2}F_{k-3}$,  
$\frac{3}{2}F_k$, type three.

Note that type one vertices are not adjacent to type two vertices.

Since $F_k+\frac{1}{2}F_{k-3}$ is the only degree 2 vertex 
 whose neighbours are only degree 2 vertices, it follows 
that  $\phi(F_k+\frac{1}{2}F_{k-3})=F_k+\frac{1}{2}F_{k-3}$. 
Vertices $\frac{1}{2}F_k$ and $\frac{3}{2}F_k$ are the only neighbours 
 of $F_k+\frac{1}{2}F_{k-3}$, so $\phi(\{\frac{1}{2}F_k,\frac{3}{2}F_k\})
      =\{\frac{1}{2}F_k,\frac{3}{2}F_k\}$.
 Thus, type three vertices are mapped to type three vertices.

Let $x \in \left(F_{k-1},F_{k}\right)$ be a type one vertex. 
Then $x$ is adjacent to $F_{k+1}-x$ (another type one vertex),
and $F_{k+1}-x \leq F_{k-1}$.

Now let $y$ be a type two vertex;  then $y$ is adjacent to $F_{k+2}-x$
 (another type two vertex) and
 $y$ is also adjacent to $F_{k+1}-y \leq F_{k-1}$. 
Vertex $F_{k+2}-x$ is adjacent to $y-F_{k} \leq F_{k-1}$. 
The sum of  $F_{k+2}-y$ and $y-F_{k}$ is $F_{k+1}$.

Type one vertices are mapped to type one vertices, 
since every  two adjacent type one
 vertices have  neighbours that sum to $F_k$. 
 
 Type two vertices are mapped to type two vertices, 
 since  two adjacent type two vertices have 
  neighbours that sum to $F_{k+1}$. Hence,  $\phi$
  preserves the type of degree 2 vertices.

For $x \in \left(F_{k-1},F_{k}\right)$, vertex $F_k-x$ is a neighbour
of $x$. Therefore, $\phi(x)$ is
adjacent to $\phi(F_k-x)=F_k-x$. However, $x$ is the only type one vertex
 adjacent to $F_k-x$, and so $\phi(x)=x$. 

Let $y$ be a  type two vertex. Since $F_{k+1}-y$ is a neighbour of $y$,
 $\phi(y)$ is adjacent to $\phi(F_{k+1}-y)=F_{k+1}-y$. 
However, $y$ is the only type two vertex adjacent to $F_{k+1}-y$,
and so  $\phi(y)=y$.

So, $\phi$ fixes each vertex of $V(G)\backslash\{\frac{1}{2}F_k, \frac{3}{2}F_k\}$. 

Since $\frac{1}{2}F_k$ and $\frac{3}{2}F_k$ have
  the same neighbours (namely, $F_{k}+\frac{1}{2}F_{k-3}$ and $\frac{1}{2}F_{k-3}$)
   either $\phi$ is the identity or $\phi$ interchanges
 $\frac{1}{2}F_k$ and $\frac{3}{2}F_k$ and fixes all other vertices.
 \bigskip

\noindent{\sc Case} 2b:  $N \in \left[F_{k+1}, \frac{1}{2}F_{k+3} \right)$.

\begin{center}
\begin{tikzpicture}[scale=.85]
\draw[black, thick] (-6.4,0) -- (6.4,0);
\filldraw[black] (-6.4,0) circle (1pt) node[anchor=south] {$F_{k-5}$};
\filldraw[black] (-3.2,0) circle (1pt) node[anchor=south] {$F_{k-4}$};
\filldraw[black] (0,0) circle (1pt) node[anchor=south] {$F_{k-3}$};
\filldraw[black] (3.2,0) circle (1pt) node[anchor=south] {$F_{k-2}$};
\filldraw[black] (6.4,0) circle (1pt) node[anchor=south] {$F_{k-1}$};
\filldraw[black] (-4.8,0) circle (1pt) node[anchor=south] {$\frac{F_{k-3}}{2}$};
\filldraw[black] (4.8,0) circle (1pt) node[anchor=south] {$\frac{F_{k}}{2}$};

\draw[black, thick] (-6.4,-2) -- (6.4,-2);
\filldraw[black] (-6.4,-2) circle (1pt) node[anchor=south] {$F_{k-1}$};
\filldraw[black] (-3.2,-2) circle (1pt) node[anchor=south] {$F_{k}$};
\filldraw[black] (0,-2) circle (1pt) node[anchor=south] {$F_{k+1}$};
\filldraw[black] (3.2,-2) circle (1pt) node[anchor=south] {$F_{k+2}$};
\filldraw[black] (6.4,-2) circle (1pt) node[anchor=south] {$F_{k+3}$};
\filldraw[black] (1.6,-2) circle (1pt) node[anchor=south] {$\frac{F_{k+3}}{2}$};

\filldraw[black] (1.2,-2) circle (1pt) node[anchor=north] {1};
\filldraw[red] (1.2,-1.2) circle (0pt) node[anchor=south] {$N$};
\draw[dotted, thick] (1.2,-1.2)--(1.2,-2);
\foreach \x in {1,...,6}
\filldraw[black] (1.2-0.2*\x,-2) circle (0.5pt) node[anchor=north] {1};

\filldraw[black] (-0.8,-2) circle (1pt) node[anchor=north] {2};
\filldraw[black] (-0.8,-2) circle (1pt) node[anchor=south] {$\frac{3F_k}{2}$};

\foreach \x in {1,...,16}
\filldraw[black] (0-0.2*\x,-2) circle (0.5pt) node[anchor=north] {2};

\filldraw[red] (-5.6,-1.2) circle (0pt) node[anchor=south] {$F_{k+2}-N$};
\draw[dotted, thick] (-5.6,-1.2) -- (-5.6,-2);

\filldraw[black] (-2.2,-2) circle (1pt) node[anchor=north] {2};
\filldraw[black] (-2.2,-2) circle (1pt) node[anchor=south] {$\frac{2F_k+F_{k-3}}{2}$};

\foreach \x in {1,...,12}
\filldraw[black] (-3.2-0.2*\x,-2) circle (0.5pt) node[anchor=north] {3};

\foreach \x in {1,...,4}
\filldraw[black] (-5.6-0.2*\x,-2) circle (0.5pt) node[anchor=north] {2};

\filldraw[black] (6.4,0) circle (0.5pt) node[anchor=north] {2};

\foreach \x in {1,...,7}
\filldraw[black] (6.4-0.2*\x,0) circle (0.5pt) node[anchor=north] {3};

\filldraw[black] (4.8,0) circle (0.5pt) node[anchor=north] {2};

\foreach \x in {1,...,8}
\filldraw[black] (4.8-0.2*\x,0) circle (0.5pt) node[anchor=north] {3};

\foreach \x in {1,...,16}
\filldraw[black] (3.2-0.2*\x,0) circle (0.5pt) node[anchor=north] {4};

\foreach \x in {1,...,16}
\filldraw[black] (0-0.2*\x,0) circle (0.5pt) node[anchor=north] {5};

\foreach \x in {1,...,7}
\filldraw[black] (-3.2-0.2*\x,0) circle (0.5pt) node[anchor=north] {6};

\filldraw[black] (-4.8,0) circle (0.5pt) node[anchor=north] {5};

\foreach \x in {1,...,8}
\filldraw[black] (-4.8-0.2*\x,0) circle (0.5pt) node[anchor=north] {6};
\end{tikzpicture}
\end{center}
\medskip

Let $S=\{1,...,F_{k+1}-1\}$. Then $S$ is a set of vertices of degree at least 2.
Since any automorphism maps 
 vertices of degree at least 2 to vertices of degree at least 2, 
 $\phi(S)=S$ and $\psi=\phi |_S $ is an automorphism of $G_{F_{k+1}-1}$. By Case 2a,
  $\psi$ is either the identity isomorphism or $\psi$  interchanges
  $\frac{3}{2}F_{k}$ and $\frac{1}{2}F_{k}$ and fixes all other vertices of $S$.

The only degree 1 vertex that has a  neighbour of degree 2
is $F_{k+1}$, and so   $\phi(F_{k+1})=F_{k+1}$. 

Consider $y > F_{k+1}$; then the degree of $y$ is 1,
 with its only neighbour  $F_{k+2}-y$. Then $\frac{1}{2}F_k<F_{k+2}-y\leq F_{k}$ 
 and therefore
  $\phi(F_{k+2}-y)=F_{k+2}-y$. Since $y$ is adjacent to $F_{k+2}-y$, then $\phi(y)$ is
  adjacent to $F_{k+2}-y$. However,  $y$ is the only degree 1 vertex  adjacent to
   $F_{k+2}-y$, and so  $\phi(y)=y$.
So,  $\phi$ fixes each vertex in $V(G_N)\backslash \{\frac{1}{2}F_k, \frac{3}{2}F_k\}$.

  Since $\frac{1}{2}F_k$ and $\frac{3}{2}F_k$
 have the same set of neighbours (both are adjacent to $F_{k}+\frac{1}{2}F_{k-3}$ and
 $\frac{1}{2}F_{k-3}$), $\phi$ either fixes both $\frac{1}{2}F_k$ and $\frac{3}{2}F_k$
 (in which case $\phi$ is the identity) or $\phi$ interchanges
 $\frac{1}{2}F_k$ and $\frac{3}{2}F_k$ (and fixes all other vertices).\bigskip

\noindent{\sc Case 3:}  $N \in \left[F_{k+1}+\frac{1}{2}F_k,F_{k+2}+\frac{1}{2}F_{k-3}\right)$. 
Note that $F_{k+1}+\frac{1}{2}F_k=\frac{1}{2}F_{k+3}$.
There are two subcases to consider:
\medskip

{\sc Case} 3a: $N \in \left[F_{k+1}+\frac{1}{2}F_k, F_{k+2} \right)$.

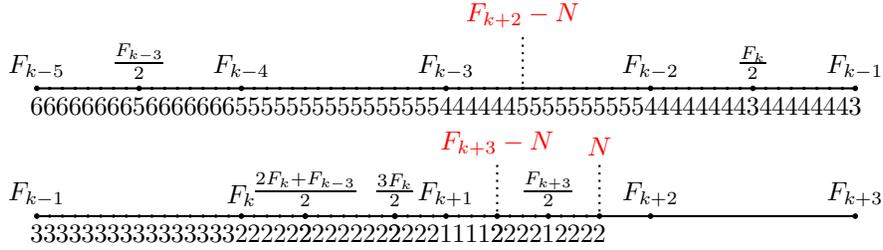
\begin{center}
\begin{tikzpicture}[scale=.85]
\draw[black, thick] (-6.4,0) -- (6.4,0);
\filldraw[black] (-6.4,0) circle (1pt) node[anchor=south] {$F_{k-5}$};
\filldraw[black] (-3.2,0) circle (1pt) node[anchor=south] {$F_{k-4}$};
\filldraw[black] (0,0) circle (1pt) node[anchor=south] {$F_{k-3}$};
\filldraw[black] (3.2,0) circle (1pt) node[anchor=south] {$F_{k-2}$};
\filldraw[black] (6.4,0) circle (1pt) node[anchor=south] {$F_{k-1}$};
\filldraw[black] (-4.8,0) circle (1pt) node[anchor=south] {$\frac{F_{k-3}}{2}$};
\filldraw[black] (4.8,0) circle (1pt) node[anchor=south] {$\frac{F_{k}}{2}$};

\draw[black, thick] (-6.4,-2) -- (6.4,-2);
\filldraw[black] (-6.4,-2) circle (1pt) node[anchor=south] {$F_{k-1}$};
\filldraw[black] (-3.2,-2) circle (1pt) node[anchor=south] {$F_{k}$};
\filldraw[black] (0,-2) circle (1pt) node[anchor=south] {$F_{k+1}$};
\filldraw[black] (3.2,-2) circle (1pt) node[anchor=south] {$F_{k+2}$};
\filldraw[black] (6.4,-2) circle (1pt) node[anchor=south] {$F_{k+3}$};
\filldraw[black] (1.6,-2) circle (1pt) node[anchor=south] {$\frac{F_{k+3}}{2}$};

\filldraw[black] (2.4,-2) circle (1pt) node[anchor=north] {2};
\filldraw[red] (2.4,-1.2) circle (0pt) node[anchor=south] {$N$};
\draw[dotted, thick] (2.4,-1.2)--(2.4,-2);

\foreach \x in {1,...,3}
\filldraw[black] (2.4-0.2*\x,-2) circle (0.5pt) node[anchor=north] {2};
\filldraw[black] (1.6,-2) circle (0.5pt) node[anchor=north] {1};

\foreach \x in {1,...,4}
\filldraw[black] (1.6-0.2*\x,-2) circle (0.5pt) node[anchor=north] {2};
\filldraw[black] (0.8,-2) circle (1pt) node[anchor=north] {2};
\filldraw[red] (0.8,-1.2) circle (0pt) node[anchor=south] {$F_{k+3}-N$};

\draw[dotted, thick] (0.8,-1.2)--(0.8,-2);

\foreach \x in {1,...,4}
\filldraw[black] (0.8-0.2*\x,-2) circle (0.5pt) node[anchor=north] {1};
\filldraw[black] (-0.8,-2) circle (1pt) node[anchor=north] {2};
\filldraw[black] (-0.8,-2) circle (1pt) node[anchor=south] {$\frac{3F_k}{2}$};

\foreach \x in {1,...,16}
\filldraw[black] (0-0.2*\x,-2) circle (0.5pt) node[anchor=north] {2};
\filldraw[black] (-2.2,-2) circle (1pt) node[anchor=north] {2};
\filldraw[black] (-2.2,-2) circle (1pt) node[anchor=south] {$\frac{2F_k+F_{k-3}}{2}$};

\foreach \x in {1,...,16}
\filldraw[black] (-3.2-0.2*\x,-2) circle (0.5pt) node[anchor=north] {3};
\filldraw[black] (6.4,0) circle (0.5pt) node[anchor=north] {3};

\foreach \x in {1,...,7}
\filldraw[black] (6.4-0.2*\x,0) circle (0.5pt) node[anchor=north] {4};
\filldraw[black] (4.8,0) circle (0.5pt) node[anchor=north] {3};

\foreach \x in {1,...,8}
\filldraw[black] (4.8-0.2*\x,0) circle (0.5pt) node[anchor=north] {4};

\foreach \x in {1,...,10}
\filldraw[black] (3.2-0.2*\x,0) circle (0.5pt) node[anchor=north] {5};

\filldraw[red] (1.2, 0.8) circle (0pt) node[anchor=south] {$F_{k+2}-N$};
\draw[dotted, thick] (1.2, 0.8) -- (1.2,0);

\foreach \x in {1,...,6}
\filldraw[black] (1.2-0.2*\x,0) circle (0.5pt) node[anchor=north] {4};

\foreach \x in {1,...,16}
\filldraw[black] (0-0.2*\x,0) circle (0.5pt) node[anchor=north] {5};

\foreach \x in {1,...,7}
\filldraw[black] (-3.2-0.2*\x,0) circle (0.5pt) node[anchor=north] {6};

\filldraw[black] (-4.8,0) circle (0.5pt) node[anchor=north] {5};

\foreach \x in {1,...,8}
\filldraw[black] (-4.8-0.2*\x,0) circle (0.5pt) node[anchor=north] {6};
\end{tikzpicture}
\end{center}
\medskip

Let $S=\{1,...,F_{k}-1\}$. Since $S$ is a set of vertices of degree at least 3 and
any automorphism maps vertices of degree at least 3  to  vertices of degree at least 3,
it follows that
$\phi(S)=S$ and the map $\psi=\phi |_S$ is an automorphism of 
$G_{F_{k}-1}$. By the induction
hypothesis, $\psi$ is the identity isomorphism.
Therefore, for all $x<F_k$,  $\phi(x)=x$.

The only degree 2 vertex that does not have a neighbour of degree 2
is $F_k$, and so $\phi(F_k)=F_k$. As in Case 1c, $\phi$ fixes all vertices of degree 2. 

 Consider a degree 1 vertex $y \geq F_{k+1}$. The only neighbour 
 of $y$ is $F_{k+2}-y\leq F_{k}$ and so $\phi(F_{k+2}-y)=F_{k+2}-y$. 
Thus $\phi(y)$ is adjacent to $F_{k+2}-y$. However, $y$ is the only degree 1
 vertex adjacent to $F_{k+2}-y$ and so $\phi(y)=y$.

%

Therefore, the only automorphism of $G_N$ is the identity. \bigskip

{\sc Case} 3b: $N \in \left[F_{k+2}, F_{k+2}+\frac{1}{2}F_{k-3}\right)$. 

\begin{center}
\begin{tikzpicture}[scale=.85]
\draw[black, thick] (-6.4,0) -- (6.4,0);
\filldraw[black] (-6.4,0) circle (1pt) node[anchor=south] {$F_{k-5}$};
\filldraw[black] (-3.2,0) circle (1pt) node[anchor=south] {$F_{k-4}$};
\filldraw[black] (0,0) circle (1pt) node[anchor=south] {$F_{k-3}$};
\filldraw[black] (3.2,0) circle (1pt) node[anchor=south] {$F_{k-2}$};
\filldraw[black] (6.4,0) circle (1pt) node[anchor=south] {$F_{k-1}$};
\filldraw[black] (-4.8,0) circle (1pt) node[anchor=south] {$\frac{F_{k-3}}{2}$};
\filldraw[black] (4.8,0) circle (1pt) node[anchor=south] {$\frac{F_{k}}{2}$};

\draw[black, thick] (-6.4,-2) -- (6.4,-2);
\filldraw[black] (-6.4,-2) circle (1pt) node[anchor=south] {$F_{k-1}$};
\filldraw[black] (-3.2,-2) circle (1pt) node[anchor=south] {$F_{k}$};
\filldraw[black] (0,-2) circle (1pt) node[anchor=south] {$F_{k+1}$};
\filldraw[black] (3.2,-2) circle (1pt) node[anchor=south] {$F_{k+2}$};
\filldraw[black] (6.4,-2) circle (1pt) node[anchor=south] {$F_{k+3}$};
\filldraw[black] (1.6,-2) circle (1pt) node[anchor=south] {$\frac{F_{k+3}}{2}$};

\filldraw[black] (4.0,-1.9) circle (0pt) node[anchor=south] {\scalebox{0.8}{$\frac{2F_{k+2}+F_{k-3}}{2}$}};
\filldraw[black] (4.0,-2) circle (1pt) {};

\filldraw[black] (3.8,-2) circle (1pt) node[anchor=north] {1};
\filldraw[red] (3.8,-1.2) circle (0pt) node[anchor=south] {$N$};
\draw[dotted, thick] (3.8,-1.2)--(3.8,-2);

\foreach \x in {1,...,3}
\filldraw[black] (3.8-0.2*\x,-2) circle (0.5pt) node[anchor=north] {1};

\foreach \x in {1,...,7}
\filldraw[black] (3.2-0.2*\x,-2) circle (0.5pt) node[anchor=north] {2};

\filldraw[black] (1.6,-2) circle (0.5pt) node[anchor=north] {1};

\foreach \x in {1,...,8}
\filldraw[black] (1.6-0.2*\x,-2) circle (0.5pt) node[anchor=north] {2};

\filldraw[red] (-0.6,-1.2) circle (0pt) node[anchor=south] {$F_{k+3}-N$};
\draw[dotted, thick] (-0.6,-1.2)--(-0.6,-2);
\foreach \x in {1,...,3}
\filldraw[black] (0-0.2*\x,-2) circle (0.5pt) node[anchor=north] {3};

\filldraw[black] (-0.8,-2) circle (1pt) node[anchor=north] {2};
\filldraw[black] (-0.8,-2) circle (1pt) node[anchor=south] {$\frac{3F_k}{2}$};

\foreach \x in {1,...,12}
\filldraw[black] (-0.8-0.2*\x,-2) circle (0.5pt) node[anchor=north] {2};

\filldraw[black] (-2.2,-2) circle (1pt) node[anchor=north] {2};
\filldraw[black] (-2.2,-2) circle (1pt) node[anchor=south] {$\frac{2F_k+F_{k-3}}{2}$};

\foreach \x in {1,...,16}
\filldraw[black] (-3.2-0.2*\x,-2) circle (0.5pt) node[anchor=north] {3};

\filldraw[black] (6.4,0) circle (0.5pt) node[anchor=north] {3};

\foreach \x in {1,...,7}
\filldraw[black] (6.4-0.2*\x,0) circle (0.5pt) node[anchor=north] {4};

\filldraw[black] (4.8,0) circle (0.5pt) node[anchor=north] {3};

\foreach \x in {1,...,8}
\filldraw[black] (4.8-0.2*\x,0) circle (0.5pt) node[anchor=north] {4};

\foreach \x in {1,...,16}
\filldraw[black] (3.2-0.2*\x,0) circle (0.5pt) node[anchor=north] {5};

\foreach \x in {1,...,16}
\filldraw[black] (0-0.2*\x,0) circle (0.5pt) node[anchor=north] {6};

\foreach \x in {1,...,7}
\filldraw[black] (-3.2-0.2*\x,0) circle (0.5pt) node[anchor=north] {7};

\filldraw[black] (-4.8,0) circle (0.5pt) node[anchor=north] {6};

\foreach \x in {1,...,8}
\filldraw[black] (-4.8-0.2*\x,0) circle (0.5pt) node[anchor=north] {7};
\end{tikzpicture}
\end{center}
\medskip

Let $S=\{1,...,F_{k+2}-1\}$. Then $S$ is a set of vertices of degree at least 2 together
with $\frac{1}{2}F_{k+3}$. Any automorphism maps vertices of degree at least 2 
to  vertices  of degree at least 2. The remaining vertex $\frac{1}{2}F_{k+3}$ is the
 only degree 1 vertex with a neighbour (namely $\frac{1}{2}F_k$) of degree 3 that
  has neighbours of degree 1, 2 and 6.
  All other degree 1 vertices have at least one 
 second neighbour of degree 7 or higher, and so
 $\phi(\frac{1}{2}F_{k+3})=\frac{1}{2}F_{k+3}$. So $\phi(S)=S$ 
 and $\psi=\phi |_S $ is an automorphism of $G_{F_{k+2}-1}$. 
 By Case 3a, $\psi$ is the identity isomorphism.
 Therefore, $\phi$ fixes each $x<F_{k+2}$.

Let $y \geq F_{k+2}$. The degree of $y$ is 1, with neighbour $F_{k+3}-y$.
Then $F_{k+3}-y\leq F_{k+1}$ and so  $\phi(F_{k+3}-y)=F_{k+3}-y$.
Since  $y$ is adjacent to $F_{k+3}-y$, it follows that $\phi(y)$ is adjacent to $F_{k+3}-y$.
However,  $y$ is the only degree 1 vertex adjacent to $F_{k+3}-y$, and so
 $\phi(y)=y$.

Therefore, the only automorphism of $G_N$ is the identity. \bigskip

\noindent{\sc Case} 4: $N \in \left[F_{k+2}+\frac{1}{2}F_{k-3}, F_{k+2}+\frac{1}{2}F_{k+1}\right)$. 
There are three subcases to consider.\medskip

{\sc Case} 4a: $N \in \left[F_{k+2}+\frac{1}{2}F_{k-3}, F_{k+2}+F_{k-1}\right)$.

\begin{center}
\begin{tikzpicture}[scale=.85]
\draw[black, thick] (-6.4,0) -- (6.4,0);
\filldraw[black] (-6.4,0) circle (1pt) node[anchor=south] {$F_{k-5}$};
\filldraw[black] (-3.2,0) circle (1pt) node[anchor=south] {$F_{k-4}$};
\filldraw[black] (0,0) circle (1pt) node[anchor=south] {$F_{k-3}$};
\filldraw[black] (3.2,0) circle (1pt) node[anchor=south] {$F_{k-2}$};
\filldraw[black] (6.4,0) circle (1pt) node[anchor=south] {$F_{k-1}$};
\filldraw[black] (-4.8,0) circle (1pt) node[anchor=south] {$\frac{F_{k-3}}{2}$};
\filldraw[black] (4.8,0) circle (1pt) node[anchor=south] {$\frac{F_{k}}{2}$};

\draw[black, thick] (-6.4,-2) -- (6.4,-2);
\filldraw[black] (-6.4,-2) circle (1pt) node[anchor=south] {$F_{k-1}$};
\filldraw[black] (-3.2,-2) circle (1pt) node[anchor=south] {$F_{k}$};
\filldraw[black] (0,-2) circle (1pt) node[anchor=south] {$F_{k+1}$};
\filldraw[black] (3.2,-2) circle (1pt) node[anchor=south] {$F_{k+2}$};
\filldraw[black] (6.4,-2) circle (1pt) node[anchor=south] {$F_{k+3}$};
\filldraw[black] (1.6,-2) circle (1pt) node[anchor=south] {$\frac{F_{k+3}}{2}$};

\filldraw[black] (4.0,-1.9) circle (0pt) node[anchor=south] {\scalebox{0.8}{$\frac{2F_{k+2}+F_{k-3}}{2}$}};
\filldraw[black] (4.0,-2) circle (1pt) {};

\filldraw[black] (4.4,-2) circle (1pt) node[anchor=north] {1};
\filldraw[red] (4.4,-1.2) circle (0pt) node[anchor=south] {$N$};
\draw[dotted, thick] (4.4,-1.2)--(4.4,-2);

\foreach \x in {1,...,6}
\filldraw[black] (4.4-0.2*\x,-2) circle (0.5pt) node[anchor=north] {1};

\foreach \x in {1,...,7}
\filldraw[black] (3.2-0.2*\x,-2) circle (0.5pt) node[anchor=north] {2};

\filldraw[black] (1.6,-2) circle (0.5pt) node[anchor=north] {1};

\foreach \x in {1,...,8}
\filldraw[black] (1.6-0.2*\x,-2) circle (0.5pt) node[anchor=north] {2};

\filldraw[red] (-2.4,-1.2) circle (0pt) node[anchor=south] {$F_{k+3}-N$};
\draw[dotted, thick] (-2.4,-1.2)--(-2.4,-2);

\foreach \x in {1,...,12}
\filldraw[black] (0-0.2*\x,-2) circle (0.5pt) node[anchor=north] {3};

\filldraw[black] (-0.8,-2) circle (1pt) node[anchor=south] {$\frac{3F_k}{2}$};

\foreach \x in {1,...,4}
\filldraw[black] (-2.4-0.2*\x,-2) circle (0.5pt) node[anchor=north] {2};

\filldraw[black] (-2.2,-2) circle (1pt) node[anchor=south] {$\frac{2F_k+F_{k-3}}{2}$};

\foreach \x in {1,...,16}
\filldraw[black] (-3.2-0.2*\x,-2) circle (0.5pt) node[anchor=north] {3};

\filldraw[black] (6.4,0) circle (0.5pt) node[anchor=north] {3};

\foreach \x in {1,...,7}
\filldraw[black] (6.4-0.2*\x,0) circle (0.5pt) node[anchor=north] {4};

\filldraw[black] (4.8,0) circle (0.5pt) node[anchor=north] {3};

\foreach \x in {1,...,8}
\filldraw[black] (4.8-0.2*\x,0) circle (0.5pt) node[anchor=north] {4};

\foreach \x in {1,...,16}
\filldraw[black] (3.2-0.2*\x,0) circle (0.5pt) node[anchor=north] {5};

\foreach \x in {1,...,16}
\filldraw[black] (0-0.2*\x,0) circle (0.5pt) node[anchor=north] {6};

\foreach \x in {1,...,7}
\filldraw[black] (-3.2-0.2*\x,0) circle (0.5pt) node[anchor=north] {7};

\filldraw[black] (-4.8,0) circle (0.5pt) node[anchor=north] {6};

\foreach \x in {1,...,8}
\filldraw[black] (-4.8-0.2*\x,0) circle (0.5pt) node[anchor=north] {7};
\end{tikzpicture}
\end{center}
\medskip

Let $S=\{1,...,F_{k-2}-1\}$. Then $S$ is the set of vertices of degree at least 5 
and the vertices of degree at least 5 are mapped to the vertices of degree 
at least 5. $\phi(S)=S$ and $\psi=\phi \mid_S $ is an automorphism of 
$G_{F_{k-2}-1}$. By the induction hypothesis, $\psi$ is either the
 identity  or the map that interchanges $\frac{1}{2}F_{k-3}$ 
 with $\frac{3}{2}F_{k-3}$ (and fixes all other vertices of $S$). 
However,  $\frac{1}{2}F_{k-3}$ and $\frac{3}{2}F_{k-3}$ have degrees 6 and 5 respectively;
so $\psi(\frac{1}{2}F_{k-3})=\frac{1}{2}F_{k-3}$ and 
 $\psi(\frac{3}{2}F_{k-3})=\frac{3}{2}F_{k-3}$. 
 Therefore, for each $x<F_{k-2}$, $\phi$ fixes $x$.
 
Let $y\in[F_{k-2},F_{k-1})$ be a vertex of degree 4. 
Then $y$ is adjacent to $F_{k-1}-y$. Vertex $F_{k-1}-y \leq F_{k-3}$. 
Then $\phi(F_{k-1}-y)=F_{k-1}-y$ is adjacent to $\phi(y)$. 
Vertex $y$ is the only vertex of degree 4 adjacent
to $F_{k-1}-y$. Therefore, $\phi$ fixes each 
$y\in[F_{k-2},F_{k-1})\backslash \{\frac{1}{2}F_k\}$.

The interval $[F_{k-1},F_k)$ consists of vertices of degree 3 that do not have 
a neighbour of degree 1. If a vertex of degree 3 has a neighbour of degree 1, 
then this vertex does not belong to interval $[F_{k-1},F_k)$. 
Therefore, $\phi([F_{k-1},F_k))=[F_{k-1},F_k)$.
 Let $y\in [F_{k-1},F_k)$; then y is adjacent to $F_k-y\leq F_{k-2}$. 
 Therefore, $\phi(y)$ is adjacent to $\phi(F_k-y)=F_k-y$. 
 Vertex $y$ is the only degree 3 vertex without a degree 1 neighbour 
  adjacent to $F_k-y$. Hence, $\phi$ fixes every vertex in $ [F_{k-1},F_k)$.

Vertex $F_{k+1}$ is the only degree 2 vertex that has a neighbour of degree 1, and
so  $\phi(F_{k+1})=F_{k+1}$. Vertex $F_{k}$ is the only degree 2 vertex
 that is adjacent to $F_{k+1}$, and so $\phi(F_k)=F_k$.
 
Proceed as in Case 1c, defining two types of degree 2 vertices, those in the interval 
$\left(F_{k}, F_{k+1}\right)$ and those in the interval $\left(F_{k+1}, F_{k+2}\right)$. 
As in Case 1c, conclude that for any vertex $y$ of degree 2,  $\phi(y)=y$.

The set  $S_1=\{1,...,F_{k+1}-1\}$ is the set of degree 3 or higher vertices
together with
 vertices of degree 2 in interval $[F_k, F_{k+1})$. Degree 3 or higher vertices are 
 mapped by any automorphism to degree 3 or higher vertices and degree 2 vertices 
 in $[F_k, F_{k+1})$ are fixed by $\phi$. 
The restriction $\psi=\phi |_{S_1} $ is an automorphism of $G_{F_{k+1}-1}$.
 By Case 2a, $\psi$ is either the identity or the isomorphism that interchanges
  only $\frac{1}{2}F_k$ with $\frac{3}{2}F_k$. 
Therefore, $\phi(x)=x$ for all vertices $x<F_{k+1}$, except  possibly 
$x=\frac{1}{2}F_k$ or $x=\frac{3}{2}F_k$.\bigskip

{\sc Case} 4b: $N \in \left[F_{k+2}+F_{k-1}, F_{k+2}+\frac{1}{2}F_{k+1}\right)$.

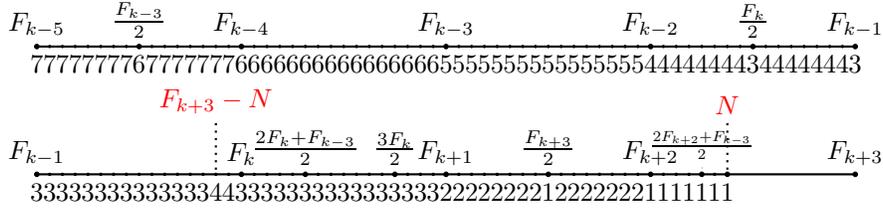
\begin{center}
\begin{tikzpicture}[scale=.85]
\draw[black, thick] (-6.4,0) -- (6.4,0);
\filldraw[black] (-6.4,0) circle (1pt) node[anchor=south] {$F_{k-5}$};
\filldraw[black] (-3.2,0) circle (1pt) node[anchor=south] {$F_{k-4}$};
\filldraw[black] (0,0) circle (1pt) node[anchor=south] {$F_{k-3}$};
\filldraw[black] (3.2,0) circle (1pt) node[anchor=south] {$F_{k-2}$};
\filldraw[black] (6.4,0) circle (1pt) node[anchor=south] {$F_{k-1}$};
\filldraw[black] (-4.8,0) circle (1pt) node[anchor=south] {$\frac{F_{k-3}}{2}$};
\filldraw[black] (4.8,0) circle (1pt) node[anchor=south] {$\frac{F_{k}}{2}$};

\draw[black, thick] (-6.4,-2) -- (6.4,-2);
\filldraw[black] (-6.4,-2) circle (1pt) node[anchor=south] {$F_{k-1}$};
\filldraw[black] (-3.2,-2) circle (1pt) node[anchor=south] {$F_{k}$};
\filldraw[black] (0,-2) circle (1pt) node[anchor=south] {$F_{k+1}$};
\filldraw[black] (3.2,-2) circle (1pt) node[anchor=south] {$F_{k+2}$};
\filldraw[black] (6.4,-2) circle (1pt) node[anchor=south] {$F_{k+3}$};
\filldraw[black] (1.6,-2) circle (1pt) node[anchor=south] {$\frac{F_{k+3}}{2}$};

\filldraw[black] (4.0,-1.9) circle (0pt) node[anchor=south] {\scalebox{0.8}{$\frac{2F_{k+2}+F_{k-3}}{2}$}};
\filldraw[black] (4.0,-2) circle (1pt) {};

\filldraw[black] (4.4,-2) circle (1pt) node[anchor=north] {1};
\filldraw[red] (4.4,-1.2) circle (0pt) node[anchor=south] {$N$};
\draw[dotted, thick] (4.4,-1.2)--(4.4,-2);

\foreach \x in {1,...,6}
\filldraw[black] (4.4-0.2*\x,-2) circle (0.5pt) node[anchor=north] {1};

\foreach \x in {1,...,7}
\filldraw[black] (3.2-0.2*\x,-2) circle (0.5pt) node[anchor=north] {2};

\filldraw[black] (1.6,-2) circle (0.5pt) node[anchor=north] {1};

\foreach \x in {1,...,8}
\filldraw[black] (1.6-0.2*\x,-2) circle (0.5pt) node[anchor=north] {2};

\foreach \x in {1,...,16}
\filldraw[black] (0-0.2*\x,-2) circle (0.5pt) node[anchor=north] {3};

\filldraw[black] (-0.8,-2) circle (1pt) node[anchor=south] {$\frac{3F_k}{2}$};

\filldraw[black] (-2.2,-2) circle (1pt) node[anchor=south] {$\frac{2F_k+F_{k-3}}{2}$};

\foreach \x in {1,...,2}
\filldraw[black] (-3.2-0.2* \x, -2) circle (0.5pt) node[anchor=north] {4};

\filldraw[red] (-3.6,-1.2) circle (0pt) node[anchor=south] {$F_{k+3}-N$};
\draw[dotted, thick] (-3.6,-1.2)--(-3.6,-2);

\foreach \x in {1,...,14}
\filldraw[black] (-3.6-0.2*\x,-2) circle (0.5pt) node[anchor=north] {3};

\filldraw[black] (6.4,0) circle (0.5pt) node[anchor=north] {3};

\foreach \x in {1,...,7}
\filldraw[black] (6.4-0.2*\x,0) circle (0.5pt) node[anchor=north] {4};

\filldraw[black] (4.8,0) circle (0.5pt) node[anchor=north] {3};

\foreach \x in {1,...,8}
\filldraw[black] (4.8-0.2*\x,0) circle (0.5pt) node[anchor=north] {4};

\foreach \x in {1,...,16}
\filldraw[black] (3.2-0.2*\x,0) circle (0.5pt) node[anchor=north] {5};

\foreach \x in {1,...,16}
\filldraw[black] (0-0.2*\x,0) circle (0.5pt) node[anchor=north] {6};

\foreach \x in {1,...,7}
\filldraw[black] (-3.2-0.2*\x,0) circle (0.5pt) node[anchor=north] {7};

\filldraw[black] (-4.8,0) circle (0.5pt) node[anchor=north] {6};

\foreach \x in {1,...,8}
\filldraw[black] (-4.8-0.2*\x,0) circle (0.5pt) node[anchor=north] {7};
\end{tikzpicture}
\end{center}
\medskip

Let $S_1=\{1,...,F_{k+1}-1\}$. Then $S_1$ is a set of vertices of degree 
at least 3 and the vertices of degree at least 3 are mapped to the vertices of degree 
at least 3. $\phi(S_1)=S$ and $\psi=\phi |_{S_1} $ is an automorphism of 
$G_{F_{k+1}-1}$. By Case 2a, $\psi$ is either the identity isomorphism or the 
one that interchanges $\frac{1}{2}F_k$ with $\frac{3}{2}F_k$. 
Therefore, $\phi(x)=x$ for all vertices $x<F_{k+1}$, except of possibly 
$x=\frac{1}{2}F_k$ or $x=\frac{3}{2}F_k$.

Vertex $F_{k+1}$ is the only degree 2 vertex that does not have a neighbour of degree 2.
Therefore, $\phi(F_{k+1})=F_{k+1}$. Let $y \geq F_{k+2}$ be a degree 1 vertex with
$y \neq \frac{1}{2}F_{k+3}$, $y \neq F_{k+2}+\frac{1}{2}F_{k-3}$ .
Since $F_{k+3}-y\leq F_{k+1}$ it follows that $\phi(F_{k+3}-y)=F_{k+3}-y$.
Also since $y$ is adjacent to $F_{k+3}-y$, it follows that $\phi(y)$ is adjacent 
to $F_{k+3}-y$.
Then $y$ is the only degree 1 vertex that is adjacent to $F_{k+3}-y$, and so $\phi(y)=y$.

Consider $F_{k+1}<y<F_{k+2}$, such that degree of $y$ is 2 (so
$y \neq \frac{1}{2}F_{k+3}$). Then $y$ is adjacent to  $F_{k+2}-y$. 
$F_{k+2}-y < F_{k}$ and $F_{k+2}-y \neq \frac{1}{2}F_{k}$ and therefore 
$\phi(F_{k+2}-y)=F_{k+2}-y$. 
$y$ is adjacent to $F_{k+2}-y$, therefore $\phi(y)$ is adjacent to $F_{k+2}-y$. $y$ is the only degree 2 vertex that is adjacent to $F_{k+2}-y$. Hence, $\phi(y)=y$.

Consider $x \geq F_{k+2}$, such that degree of $x$ is 1 and 
$x \neq F_{k+2}+\frac{1}{2}F_{k-3}$ . Then $x$ is adjacent to 
$F_{k+3}-y$. $F_{k+3}-x < F_{k+1}$ and $F_{k+3}-x \neq \frac{1}{2}F_{k}$, 
$F_{k+3}-x \neq \frac{3}{2}F_{k}$, therefore $\phi(F_{k+3}-x)=F_{k+3}-x$.
Since $x$ is adjacent to $F_{k+3}-x$, it follows that $\phi(x)$ is adjacent to $F_{k+3}-x$. 
Since $x$ is the only degree 1 vertex  adjacent to $F_{k+3}-x$, conclude 
that  $\phi(x)=x$.

So, $\phi$ fixes all but perhaps four of the vertices of $G_N$,
with the only possible exceptions being that $\phi$ might 
interchange $\frac{1}{2}F_k$ with $\frac{3}{2}F_k$ and  
$\frac{1}{2}F_{k+3}$ with $F_{k+2}+\frac{1}{2}F_{k-3}$.
 However,  $\frac{1}{2}F_k$ is adjacent to $\frac{1}{2}F_{k+3}$ and 
 $\frac{3}{2}F_k$ is adjacent to $F_{k+2}+\frac{1}{2}F_{k-3}$,
 so $\phi$ moves $\frac{1}{2}F_k$ iff $\phi$ moves $\frac{1}{2}F_{k+3}$.

Therefore, $\mbox{Aut}(G_N)$ consists of two mappings: the identity and 
the automorphism that  interchanges $\frac{1}{2}F_k$ with $\frac{3}{2}F_k$,
 interchanges  $\frac{1}{2}F_{k+3}$ with $F_{k+2}+\frac{1}{2}F_{k-3}$, 
  and fixes all other vertices. \bigskip

{\sc Case} 4c: 
$N \in \left[F_{k+2}+\frac{1}{2}F_{k+1}, F_{k+3}-\frac{1}{2}F_{k}\right)$.

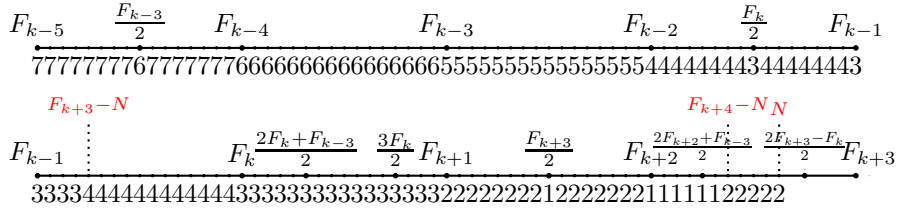
\begin{center}
\begin{tikzpicture}[scale=.85]
\draw[black, thick] (-6.4,0) -- (6.4,0);
\filldraw[black] (-6.4,0) circle (1pt) node[anchor=south] {$F_{k-5}$};
\filldraw[black] (-3.2,0) circle (1pt) node[anchor=south] {$F_{k-4}$};
\filldraw[black] (0,0) circle (1pt) node[anchor=south] {$F_{k-3}$};
\filldraw[black] (3.2,0) circle (1pt) node[anchor=south] {$F_{k-2}$};
\filldraw[black] (6.4,0) circle (1pt) node[anchor=south] {$F_{k-1}$};
\filldraw[black] (-4.8,0) circle (1pt) node[anchor=south] {$\frac{F_{k-3}}{2}$};
\filldraw[black] (4.8,0) circle (1pt) node[anchor=south] {$\frac{F_{k}}{2}$};

\draw[black, thick] (-6.4,-2) -- (6.4,-2);
\filldraw[black] (-6.4,-2) circle (1pt) node[anchor=south] {$F_{k-1}$};
\filldraw[black] (-3.2,-2) circle (1pt) node[anchor=south] {$F_{k}$};
\filldraw[black] (0,-2) circle (1pt) node[anchor=south] {$F_{k+1}$};
\filldraw[black] (3.2,-2) circle (1pt) node[anchor=south] {$F_{k+2}$};
\filldraw[black] (6.6,-2) circle (0pt) node[anchor=south] {$F_{k+3}$};
\filldraw[black] (6.4,-2) circle (1pt);
\filldraw[black] (1.6,-2) circle (1pt) node[anchor=south] {$\frac{F_{k+3}}{2}$};

\filldraw[black] (4.0,-1.9) circle (0pt) node[anchor=south] {\scalebox{0.8}{$\frac{2F_{k+2}+F_{k-3}}{2}$}};
\filldraw[black] (4.0,-2) circle (1pt);

\filldraw[black] (5.6,-1.9) circle (0pt) node[anchor=south] {\scalebox{0.8}{$\frac{2F_{k+3}-F_{k}}{2}$}};
\filldraw[black] (5.6,-2) circle (1pt) {};

\foreach \x in {1,...,5}
\filldraw[black] (5.4-0.2*\x,-2) circle (0.5pt) node[anchor=north] {2};

\filldraw[red] (5.2,-1.2) circle (0pt) node[anchor=south] {$\scriptstyle N$};
\draw[dotted, thick] (5.2,-1.2)--(5.2,-2);

\filldraw[red] (4.4,-1.2) circle (0pt) node[anchor=south] {$\scriptstyle F_{k+4}-N$};
\draw[dotted, thick] (4.4,-1.2)--(4.4,-2);

\foreach \x in {1,...,6}
\filldraw[black] (4.4-0.2*\x,-2) circle (0.5pt) node[anchor=north] {1};

\foreach \x in {1,...,7}
\filldraw[black] (3.2-0.2*\x,-2) circle (0.5pt) node[anchor=north] {2};

\filldraw[black] (1.6,-2) circle (0.5pt) node[anchor=north] {1};

\foreach \x in {1,...,8}
\filldraw[black] (1.6-0.2*\x,-2) circle (0.5pt) node[anchor=north] {2};

\foreach \x in {1,...,16}
\filldraw[black] (0-0.2*\x,-2) circle (0.5pt) node[anchor=north] {3};

\filldraw[black] (-0.8,-2) circle (1pt) node[anchor=south] {$\frac{3F_k}{2}$};

\filldraw[black] (-2.2,-2) circle (1pt) node[anchor=south] {$\frac{2F_k+F_{k-3}}{2}$};

\foreach \x in {1,...,12}
\filldraw[black] (-3.2-0.2* \x, -2) circle (0.5pt) node[anchor=north] {4};

\filldraw[red] (-5.6,-1.2) circle (0pt) node[anchor=south] {$\scriptstyle F_{k+3}-N$};
\draw[dotted, thick] (-5.6,-1.2)--(-5.6,-2);

\foreach \x in {1,...,4}
\filldraw[black] (-5.6-0.2*\x,-2) circle (0.5pt) node[anchor=north] {3};

\filldraw[black] (6.4,0) circle (0.5pt) node[anchor=north] {3};

\foreach \x in {1,...,7}
\filldraw[black] (6.4-0.2*\x,0) circle (0.5pt) node[anchor=north] {4};

\filldraw[black] (4.8,0) circle (0.5pt) node[anchor=north] {3};

\foreach \x in {1,...,8}
\filldraw[black] (4.8-0.2*\x,0) circle (0.5pt) node[anchor=north] {4};

\foreach \x in {1,...,16}
\filldraw[black] (3.2-0.2*\x,0) circle (0.5pt) node[anchor=north] {5};

\foreach \x in {1,...,16}
\filldraw[black] (0-0.2*\x,0) circle (0.5pt) node[anchor=north] {6};

\foreach \x in {1,...,7}
\filldraw[black] (-3.2-0.2*\x,0) circle (0.5pt) node[anchor=north] {7};

\filldraw[black] (-4.8,0) circle (0.5pt) node[anchor=north] {6};

\foreach \x in {1,...,8}
\filldraw[black] (-4.8-0.2*\x,0) circle (0.5pt) node[anchor=north] {7};

\end{tikzpicture}
\end{center}
\medskip

Note that $\frac{1}{2}F_{k+1}$ is not an integer.

Let $S=\{1,...,F_{k+1}-1\}$. Then $S$ is a set of vertices of degree at least 3 and
the vertices of degree at least 3 are mapped to the vertices of degree at least 3.
So $\phi(S)=S$ and $\psi=\phi \mid_S $ is an automorphism of $G_{F_{k+1}-1}$.
By the Case 2a, $\psi$ is either an identity isomorphism or one that fixes all but
two vertices, interchanging $\frac{1}{2}F_k$ with $\frac{3}{2}F_k$. 
Therefore, $\phi$ fixes each vertex in
 $[1,F_{k+1}-1]\backslash \{\frac{1}{2}F_k,\frac{3}{2}F_k\}$.

Vertex $F_{k+1}$ is the only degree 2 vertex that does not have a neighbour of degree 2.
Therefore, $\phi(F_{k+1})=F_{k+1}$. Let $y \geq F_{k+2}$ be a degree 1 vertex, such that
$y \neq F_{k+2}+\frac{1}{2}F_{k-3}$.
Vertex $y$ is adjacent to  $F_{k+3}-y$, and since $F_{k+3}-y\leq F_{k+1}$, it follows that
 $\phi(F_{k+3}-y)=F_{k+3}-y$.
 Also since $y$ is adjacent to $F_{k+3}-y$, it follows that $\phi(y)$ is 
 adjacent to $F_{k+3}-y$.
 Vertex $y$ is the only degree 1 vertex that is adjacent to $F_{k+3}-y$, and 
 therefore $\phi(y)=y$.

Proceeding as in Case 1c, with two types of degree 2 vertices
(type one are degree 2 vertices in the interval
$\left(F_{k+1}, F_{k+2}\right)$ and type two in$\left[F_{k+4}-N, N\right]$),
it follows that  $\phi$ fixes all vertices of degree 2 in $G_N$.

So, $\phi$ fixes almost all vertices of $G_N$,
with the possible exceptions that $\phi$ might interchange 
$\frac{1}{2}F_k$ with $\frac{3}{2}F_k$ 
and  interchange $\frac{1}{2}F_{k+3}$ with $F_{k+2}+\frac{1}{2}F_{k-3}$. 
Since $\frac{1}{2}F_k$ is adjacent to $\frac{1}{2}F_{k+3}$ 
and $\frac{3}{2}F_k$ is adjacent to $F_{k+2}+\frac{1}{2}F_{k-3}$, 
$\phi$ moves $\frac{1}{2}F_k$ iff $\phi$ moves $\frac{1}{2}F_{k+3}$.

Therefore, $\mbox{Aut}(G_N)$ consists of the identity and the
automorphism that
 interchanges $\frac{1}{2}F_k$ with $\frac{3}{2}F_k$, interchanges
   $\frac{1}{2}F_{k+3}$ with
 $F_{k+2}+\frac{1}{2}F_{k-3}$, and fixes all other vertices. \bigskip

\noindent{\sc Case} 5: $N \in \left[F_{k+3}-\frac{1}{2}F_k, F_{k+3}\right)$. 

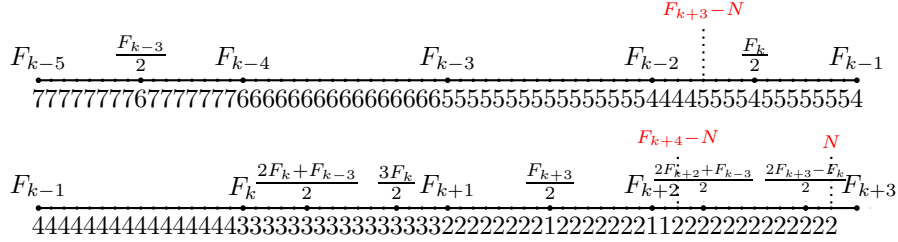
\begin{center}
\begin{tikzpicture}[scale=.85]
\draw[black, thick] (-6.4,0) -- (6.4,0);
\filldraw[black] (-6.4,0) circle (1pt) node[anchor=south] {$F_{k-5}$};
\filldraw[black] (-3.2,0) circle (1pt) node[anchor=south] {$F_{k-4}$};
\filldraw[black] (0,0) circle (1pt) node[anchor=south] {$F_{k-3}$};
\filldraw[black] (3.2,0) circle (1pt) node[anchor=south] {$F_{k-2}$};
\filldraw[black] (6.4,0) circle (1pt) node[anchor=south] {$F_{k-1}$};
\filldraw[black] (-4.8,0) circle (1pt) node[anchor=south] {$\frac{F_{k-3}}{2}$};
\filldraw[black] (4.8,0) circle (1pt) node[anchor=south] {$\frac{F_{k}}{2}$};

\draw[black, thick] (-6.4,-2) -- (6.4,-2);
\filldraw[black] (-6.4,-2) circle (1pt) node[anchor=south] {$F_{k-1}$};
\filldraw[black] (-3.2,-2) circle (1pt) node[anchor=south] {$F_{k}$};
\filldraw[black] (0,-2) circle (1pt) node[anchor=south] {$F_{k+1}$};
\filldraw[black] (3.2,-2) circle (1pt) node[anchor=south] {$F_{k+2}$};
\filldraw[black] (6.6,-2) circle (0pt) node[anchor=south] {$F_{k+3}$};
\filldraw[black] (6.4,-2) circle (1pt);
\filldraw[black] (1.6,-2) circle (1pt) node[anchor=south] {$\frac{F_{k+3}}{2}$};

\filldraw[black] (4.0,-1.9) circle (0pt) node[anchor=south] {\scalebox{0.8}{$\frac{2F_{k+2}+F_{k-3}}{2}$}};
\filldraw[black] (4.0,-2) circle (1pt);

\filldraw[black] (5.6,-1.9) circle (0pt) node[anchor=south] {\scalebox{0.8}{$\frac{2F_{k+3}-F_{k}}{2}$}};
\filldraw[black] (5.6,-2) circle (1pt) {};

\foreach \x in {1,...,13}
\filldraw[black] (6.2-0.2*\x,-2) circle (0.5pt) node[anchor=north] {2};

\filldraw[red] (6.0,-1.2) circle (0pt) node[anchor=south] {$\scriptstyle N$};
\draw[dotted, thick] (6.0,-1.2)--(6.0,-2);

\filldraw[red] (3.6,-1.2) circle (0pt) node[anchor=south] {$\scriptstyle F_{k+4}-N$};
\draw[dotted, thick] (3.6,-1.2)--(3.6,-2);

\foreach \x in {1,...,2}
\filldraw[black] (3.6-0.2*\x,-2) circle (0.5pt) node[anchor=north] {1};

\foreach \x in {1,...,7}
\filldraw[black] (3.2-0.2*\x,-2) circle (0.5pt) node[anchor=north] {2};

\filldraw[black] (1.6,-2) circle (0.5pt) node[anchor=north] {1};

\foreach \x in {1,...,8}
\filldraw[black] (1.6-0.2*\x,-2) circle (0.5pt) node[anchor=north] {2};

\foreach \x in {1,...,16}
\filldraw[black] (0-0.2*\x,-2) circle (0.5pt) node[anchor=north] {3};

\filldraw[black] (-0.8,-2) circle (1pt) node[anchor=south] {$\frac{3F_k}{2}$};

\filldraw[black] (-2.2,-2) circle (1pt) node[anchor=south] {$\frac{2F_k+F_{k-3}}{2}$};

\foreach \x in {1,...,16}
\filldraw[black] (-3.2-0.2* \x, -2) circle (0.5pt) node[anchor=north] {4};

\filldraw[black] (6.4,0) circle (0.5pt) node[anchor=north] {4};

\foreach \x in {1,...,7}
\filldraw[black] (6.4-0.2*\x,0) circle (0.5pt) node[anchor=north] {5};

\filldraw[black] (4.8,0) circle (0.5pt) node[anchor=north] {4};

\foreach \x in {1,...,4}
\filldraw[black] (4.8-0.2*\x,0) circle (0.5pt) node[anchor=north] {5};

\filldraw[red] (4.0,0.8) circle (0pt) node[anchor=south] {$\scriptstyle F_{k+3}-N$};
\draw[dotted, thick] (4.0,0.8)--(4.0,0);

\foreach \x in {1,...,4}
\filldraw[black] (4.0-0.2*\x,0) circle (0.5pt) node[anchor=north] {4};

\foreach \x in {1,...,16}
\filldraw[black] (3.2-0.2*\x,0) circle (0.5pt) node[anchor=north] {5};

\foreach \x in {1,...,16}
\filldraw[black] (0-0.2*\x,0) circle (0.5pt) node[anchor=north] {6};

\foreach \x in {1,...,7}
\filldraw[black] (-3.2-0.2*\x,0) circle (0.5pt) node[anchor=north] {7};

\filldraw[black] (-4.8,0) circle (0.5pt) node[anchor=north] {6};

\foreach \x in {1,...,8}
\filldraw[black] (-4.8-0.2*\x,0) circle (0.5pt) node[anchor=north] {7};

\end{tikzpicture}
\end{center}
\medskip

Let $S=\{1,...,F_{k+1}-1\}$. Then $S$ is a set of vertices of degree at least 3.
Since any automorphism sends the
vertices of degree at least 3 to the vertices of degree at least 3,
$\phi(S)=S$ and $\psi=\phi |_S $ is an automorphism of $G_{F_{k+1}-1}$.
 By  Case 2a, $\psi$ is either
 an identity isomorphism or one that interchanges  $\frac{1}{2}F_k$ with $\frac{3}{2}F_k$.
 However, the degrees of $\frac{1}{2}F_k$ and $\frac{3}{2}F_k$ 
 are 4 and 3 respectively, and so for all  $x<F_{k+1}$, $\phi(x)=x$.

Vertex $F_{k+1}$ is the only degree 2 vertex that does not have a neighbour 
of degree 2, and so $\phi(F_{k+1})=F_{k+1}$. 
 Vertex $\frac{1}{2}F_{k+3}$ is the only degree 1 vertex  adjacent to $\frac{1}{2}F_k$, 
 and so $\phi(\frac{1}{2}F_{k+3})=F_{k+3}$.
 
  Let $y \geq F_{k+2}$ be a degree 1 vertex.
Then $y$ is adjacent to $F_{k+3}-y$, and since $F_{k+3}-y\leq F_{k+1}$, it follows that
$\phi(F_{k+3}-y)=F_{k+3}-y$.
Also, since $y$ is adjacent to $F_{k+3}-y$, it follows that 
 $\phi(y)$ is adjacent to $F_{k+3}-y$.
However,  $y$ is the only degree 1 vertex adjacent to $F_{k+3}-y$ and so $\phi(y)=y$.
Similar to Case 1c, $\phi$ fixes all vertices of degree 2 in $G_N$.

So  $\phi$ fixes all vertices of $G_N$; that is, $\mbox{Aut}(G_N)$ consists only of 
the identity.
This completes Case 5, and hence the proof of the induction step. \bigskip

By mathematical induction,  the proof of the theorem is complete.\end{proof}

\section{Concluding remarks}\label{se:conclusion}
This work may be seen as a starting point for other graphs
formed by sums in any set given by a linear recurrence of order 2.
Some properties for such graphs are known or conjectured (see \cite{Cost:08}
and \cite{FKMOP:14} for references).   Many of the known proofs
(including some here) for Fibonacci-sum graphs may extend accordingly.

\end{document}